\documentclass[a4paper,reqno, 10pt]{amsart}

\usepackage{amsthm}
\usepackage{amssymb}
\usepackage{amsmath}
\usepackage[dvipsnames]{xcolor}
\usepackage{graphicx}
\usepackage{tikz}
\usetikzlibrary{decorations.pathreplacing,calligraphy,calc}
\usepackage{fullpage}
\usepackage[english]{babel}
\usepackage{bm}
\usepackage{txfonts}
\usepackage{comment}
\usepackage{enumitem}
\usepackage{dsfont}
    
\usepackage{microtype}
\usepackage[colorlinks = true, linkbordercolor = {white}, urlcolor={purple}, linkcolor={purple}, citecolor = {purple}]{hyperref}
\usepackage{cleveref}
\usepackage[foot]{amsaddr}
\usepackage{marginnote}
\usepackage{mathrsfs}
\usepackage{subcaption}
\usepackage{array}
\usepackage{multirow}
\usepackage{transparent}
\usepackage{float}
\usepackage{wrapfig}

\usepackage[labelfont=bf]{caption}
\numberwithin{equation}{section}

\author[S.\,van~Golden]{S.\,van~Golden}
\author[S.\,Kombrink]{S.\,Kombrink}
\author[T.\,Samuel]{T.\,Samuel}

\address[S.\,van~Golden]{Mathematisch Instituut, Universiteit Leiden, The Netherlands}
\address[S.\,Kombrink]{Fachbereich 3 Mathematik, Universit\"at Bremen, Germany}
\address[T.\,Samuel]{Department of Mathematics and Statistics, University of Exeter, UK}

\newcommand{\diam}{\operatorname{diam}}

\renewcommand{\emph}{\textsl}

\renewcommand{\textit}{\textsl}

\numberwithin{equation}{section}

\title{On the geometry of generalised Koch snowflakes}

\newtheorem{theorem}{Theorem}[section]
\newtheorem{prop}[theorem]{Proposition}
\newtheorem{lemma}[theorem]{Lemma}

\theoremstyle{definition}

\theoremstyle{remark}

\Crefname{theorem}{Theorem}{Theorems}
\Crefname{prop}{Proposition}{Propositions}
\Crefname{lemma}{Lemma}{Lemma}
\Crefname{cor}{Corollary}{Corollaries}
\Crefname{defn}{Definition}{Definitions}
\Crefname{question}{Question}{Questions}
\Crefname{remark}{Remark}{Remarks}
\Crefname{ex}{Example}{Examples}

\makeatletter
\usetikzlibrary{decorations.pathreplacing, calligraphy, calc, math, lindenmayersystems}
\usetikzlibrary{angles,quotes}
\makeatother

\begin{document}

\begin{abstract}
We consider the geometry of a class of fractal sets in $\mathbb R^2$ that generalise the famous Koch curve and Koch snowflake. While the classical Koch curve is defined by an iterative process that divides a line segment into three parts and replaces the middle part by the legs of an isosceles triangle `above' the line segment, in this more general setting, a choice can be made at each iteration as to whether to place this triangle `above' or `below' the line segment. The resulting fractals bear a striking visual resemblance to curves appearing in nature, such as coastlines and snowflakes. While these fractals can be generated by a random process that flips a coin each time to decide the orientation of the triangle, leading to `almost sure' results for their geometrical properties, we define and study them deterministically to provide exact results. In particular, we show, using the theory of non-integer expansions, that the set of all possible values for the area enclosed by these generalised Koch curves is a closed interval. Moreover, we prove that the union of all these generalised snowflakes does not contain an open set, and has zero $2$-dimensional Lebesgue measure. Complementing these results, using arguments from calculus and fractal geometry, namely properties of geometric series and Frostman's Lemma, we show that each generalised Koch curve has infinite length and the same Hausdorff dimension as its classical counterpart. Further, we also give a classification for when a generalised Koch curve is a quasicircle.
\end{abstract}

\maketitle

\section{Introduction}\label{sec1}

Fractal geometry is a growing field of mathematics, both because of its theoretical beauty and its many applications, see for instance \cite{falconer2013fractal,MR1484412,MR665254}. The \textbf{$p$-Koch curve}, for $p \in (\frac14, \frac13]$, is one of the most studied fractal sets. Like most fractal sets it is generated via an iterative process. Let $V\subset \mathbb R^2$ denote the closed filled-in rhombus with vertices 
    \begin{align}\label{eq:main_vertices}
        P_{+} = (\tfrac{1}{2},0)^{\top}, \quad
        P_{-} = (-\tfrac{1}{2},0)^{\top}, \quad
        Q_{+} = (0,\tfrac{1}{2}\!\sqrt{4p-1})^{\top} \quad \text{and} \quad
        Q_{-} = (0,-\tfrac{1}{2}\!\sqrt{4p-1})^{\top},
    \end{align}
that is, $V \subset \mathbb{R}^2$ is the closed convex hull of $P_{+}$, $P_{-}$, $Q_{+}$ and $Q_{-}$, see Figure~\ref{fig:V:first}. We define the contractive affine maps $\psi_i:V \to V$ for $i \in \{0,1,2,3\}$ by
\begin{align*}
    \psi_{0}\begin{pmatrix}
        x\\
        y
    \end{pmatrix} &= 
    \begin{pmatrix}
        p & 0\\
        0 & p
    \end{pmatrix}\begin{pmatrix}
        x\\
        y
    \end{pmatrix} - \frac12 \begin{pmatrix}
        1-p \\
        0
    \end{pmatrix},\quad  
    \psi_{1}\begin{pmatrix}
        x\\
        y
    \end{pmatrix} = \frac{1}{2}\begin{pmatrix}
        1-2p & -\sqrt{4p-1}\\
        \sqrt{4p-1} & 1-2p
        \end{pmatrix}\begin{pmatrix}
        x\\
        y
    \end{pmatrix} + \frac{1}{4}\begin{pmatrix}
        2p-1\\
        \sqrt{4p-1}
    \end{pmatrix},\\[0.5em]
    \psi_{2}\begin{pmatrix}
        x\\
        y
    \end{pmatrix} &= 
    \frac{1}{2}\begin{pmatrix}
        1-2p & \sqrt{4p-1}\\
        -\sqrt{4p-1} & 1-2p
        \end{pmatrix}\begin{pmatrix}
        x\\
        y
    \end{pmatrix} + \frac{1}{4}\begin{pmatrix}
        1-2p\\
        \sqrt{4p-1}
    \end{pmatrix} 
    \quad \text{and} \quad
    \psi_{3}\begin{pmatrix}
        x\\
        y
    \end{pmatrix} = 
    \begin{pmatrix}
        p & 0\\
        0 & p
    \end{pmatrix}\begin{pmatrix}
        x\\
        y
    \end{pmatrix} + \frac12 \begin{pmatrix}
        1-p \\
        0
    \end{pmatrix}.
\end{align*}

\tikzmath{\x = 1/(2 * sqrt(6)); \y = (7/24) * \x; \z = (1 - 7/24) * \x; } 
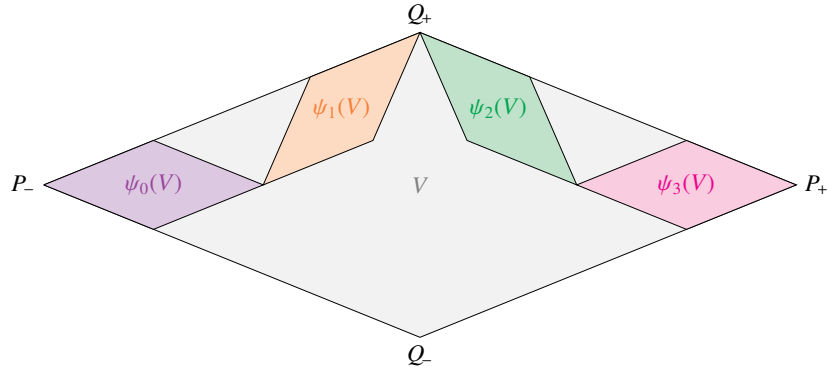
\begin{figure}[h!]
    \centering
    \scalebox{0.90}{
    \begin{tikzpicture}[scale = 11]
        \draw[black,fill=gray!10, -] (-1/2,0) -- (0,\x) -- (1/2,0) -- (0, -\x) -- (-1/2,0); 
        \node at (0,0) {\color{gray}$V$};
        \filldraw[fill=Purple!25] (-1/2,0) -- (-17/48,\y) -- (7/24 - 1/2,0) -- (-17/48, -\y) -- (-1/2,0); 
        \node at (7/48-1/2,0) {\color{Purple}$\psi_{0}(V)$};
        \filldraw[fill=Orange!25] (7/24 - 1/2,0) -- (-7/48,\z) -- (0, \x) -- (-1/16, \y) -- (7/24 - 1/2,0); 
         \node at (-5/48,\x/2) {\color{Orange}$\psi_{1}(V)$};
         \filldraw[fill=Green!25] (0, \x) -- (7/48,\z) -- (1/2 - 7/24,0) -- (1/16, \y) -- (0, \x); 
        \node at (5/48,\x/2) {\color{Green}$\psi_{2}(V)$};
         \filldraw[fill=magenta!25] (1/2 - 7/24,0) -- (17/48,\y) -- (1/2,0) -- (17/48, -\y) -- (1/2 - 7/24,0); 
        \node at (1/2-7/48,0) {\color{magenta}$\psi_{3}(V)$};
        \filldraw[black] (-1/2,0) circle (0pt) node[anchor=east] {$P_{\!-}$};
        \filldraw[black] (1/2,0) circle (0pt) node[anchor=west] {$P_{\!+}$};
        \filldraw[black] (0,\x) circle (0pt) node[anchor=south] {$Q_{\!+}$};
        \filldraw[black] (0,-\x) circle (0pt) node[anchor=north] {$Q_{\!-}$};
    \end{tikzpicture}}
    \caption{The rhombus $V$ and the images of $V$ under $\psi_{0}$, $\psi_{1}$, $\psi_{2}$ and $\psi_{3}$ for $p = \tfrac{7}{24}$.}
    \label{fig:V:first}
\end{figure}

The $p$-Koch curve is then defined to be the set
\begin{align} \label{eq : p-Koch curve}
    \bigcup_{(u_k)_{k\in \mathbb N} \in \{0,1,2,3\}^{\mathbb N}}
    \bigcap_{k\in \mathbb N} \left(\psi_{u_1}\circ \psi_{u_2} \circ \cdots \circ \psi_{u_k}\right)(V),
\end{align}
see \Cref{fig:enterlabel} for a geometric approximation.

\begin{figure}[h!]
   \centering
   \includegraphics[width=0.49\linewidth]{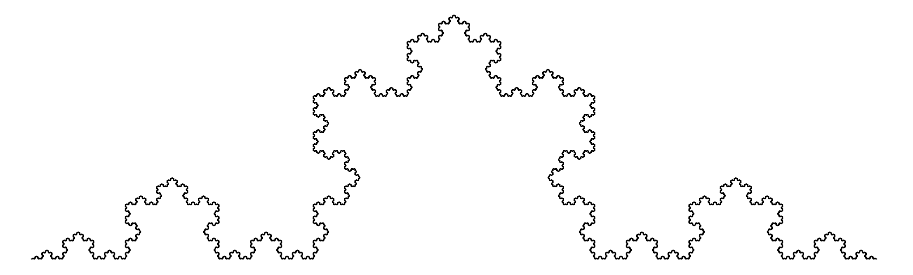}
    \includegraphics[width=0.49\linewidth]{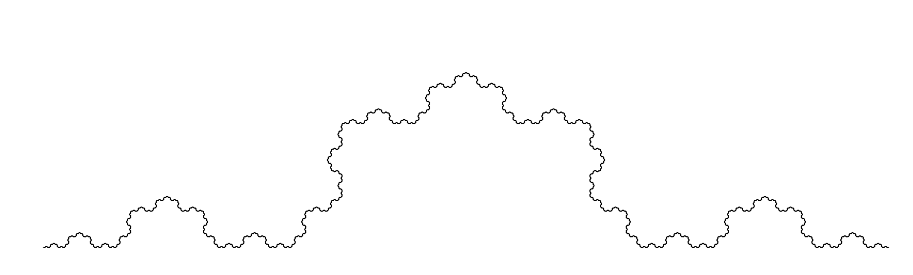}
    \caption{Approximations of the $p$-Koch curves for $p=\tfrac13$ (left) and for $p=(2+\sqrt{2})^{-1}$ (right).}
    \label{fig:enterlabel}
\end{figure}

In other words, it is the unique attractor set of the \textbf{iterated function system} $\{\psi_{0}, \psi_{1}, \psi_{2}, \psi_{3}\}$, see \cite{falconer2013fractal} for more on iterated function systems. The $p$-Koch curve is well known to be a continuous curve of infinite length in $\mathbb{R}^{2}$. The classical Koch curve (with $p=\tfrac13$) was introduced in \cite{von1906methode} as an example of a curve that is continuous everywhere but nowhere differentiable.

An alternative (but perhaps more intuitive) iterative process to construct the $p$-Koch curve using the maps $\psi_{0}, \psi_{1}, \psi_{2}$ and $\psi_{3}$ is to start with the line segment $E_0 = [-\tfrac12,\tfrac12]\,\times\,\{0\} \subset \mathbb{R}^{2}$ and to define $E_j = \bigcup_{i\in \{0,1,2,3\}} \psi_{i}(E_{j-1}) \subset \mathbb{R}^{2}$ for each $j\in \mathbb{N}$. As $j$ increases, the sets $E_j$ converge to the $p$-Koch curve with respect to the Hausdorff metric (see \cite{falconer2013fractal,MR1484412} for further details). At each iteration, each line segment is split into three segments, two outer segments of proportion $p$ and a middle segment of proportion $1-2p$, and the middle segment is replaced by two sides of an isosceles triangle with legs of proportional length $p$, see \Cref{fig:nolabel1}.

\tikzmath{ \x = sqrt(3)/2;}
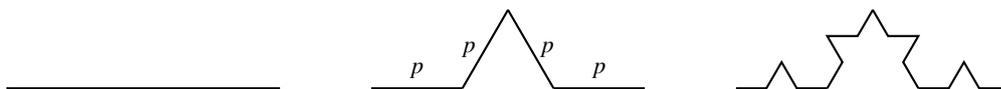
\begin{figure}[h!]
\centering
    \begin{tikzpicture}[scale = 1.2]
    \draw[black, thick, -] (0,0) -- (3,0) node[midway,below]{};
    \draw[black, thick, -] (4,0) -- (5,0) node[midway,above]{\small{$p$}};
    \draw[black, thick, -] (5,0) -- (5.5,\x);
    \draw[black, thick, -] (5.5,\x) -- (6,0);
    \draw[black, thick, -] (6,0) -- (7,0) node[midway,above]{\small{$p$}};
    \filldraw[black] (5.25,\x/2) circle (0pt) node[anchor=east] {\small{$p$}};
    \filldraw[black] (5.75,\x/2) circle (0pt) node[anchor=west] {\small{$p$}};
    \draw[black, thick, -] (8,0) -- (8 + 1/3,0) -- (8.5,\x/3) -- (8 + 2/3, 0) -- (9,0);
    \draw[black, thick, -] (9,0) -- (9+1/6,\x/3) -- (9,\x/3+\x/3) -- (9.5-1/6,\x - \x/3) -- (9.5,\x);
    \draw[black, thick, -] (9.5,\x) -- (9.5+1/6,\x - \x/3) -- (10, \x/3+\x/3) -- (10-1/6,\x/3) -- (10,0);
    \draw[black, thick, -] (10,0) -- (10 + 1/3,0) -- (10.5,\x/3) -- (10 + 2/3, 0) --  (11,0);    
    \end{tikzpicture}
    \caption{The curves $E_0$, $E_1$ and $E_2$ for $p=\tfrac13$.}
    \label{fig:nolabel1}
\end{figure}

In this construction, the orientation of the triangles that replace the middle parts of each line segment is fixed. 
In this article, we consider \textbf{generalised $p$-Koch curves}, where for each line segment in each iterative step a choice is made for the direction the triangle is pointed. To formalise this, we set $\phi_{(i,0)} = \psi_{i}$, for $i\in \{0,1,2,3\}$, and define the maps $\phi_{(i,1)}:V \to V$ for $i\in \{0,1,2,3\}$ by $\phi_{(0,1)} = \psi_{0}$, $\phi_{(3,1)} = \psi_{3}$,
\begin{align*}
    \phi_{(1,1)}\begin{pmatrix}
        x\\
        y
    \end{pmatrix} &= \frac{1}{2}\begin{pmatrix}
        1-2p & \sqrt{4p-1}\\
        -\sqrt{4p-1} & 1-2p
        \end{pmatrix}\begin{pmatrix}
        x\\
        y
    \end{pmatrix} + \frac{1}{4}\begin{pmatrix}
        2p-1\\
        -\sqrt{4p-1}
    \end{pmatrix}
 \qquad \text{and}\\[0.5em]
    \phi_{(2,1)}\begin{pmatrix}
        x\\
        y
    \end{pmatrix} &= 
    \frac{1}{2}\begin{pmatrix}
        1-2p & -\sqrt{4p-1}\\
        \sqrt{4p-1} & 1-2p
        \end{pmatrix}\begin{pmatrix}
        x\\
        y
    \end{pmatrix} + \frac{1}{4}\begin{pmatrix}
        1-2p\\
        -\sqrt{4p-1}
    \end{pmatrix},
\end{align*}
see \Cref{Fig:Robus-in}.

\tikzmath{\x = 1/(2 * sqrt(6)); \y = (7/24) * \x; \z = (1 - 7/24) * \x; } 
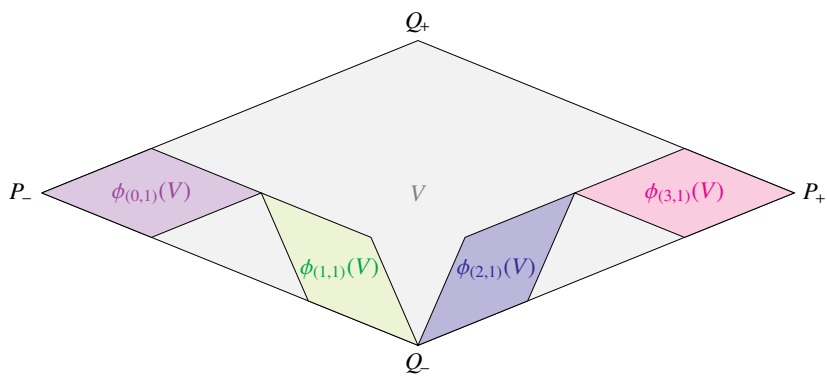
\begin{figure}[h!]
    \centering
    \scalebox{0.90}{
    \begin{tikzpicture}[scale = 11]
        \draw[black,fill=gray!10,-] (-1/2,0) -- (0,\x) -- (1/2,0) -- (0, -\x) -- (-1/2,0); 
        \node at (0,0) {\color{gray}$V$};
        \filldraw[fill=Purple!25] (-1/2,0) -- (-17/48,\y) -- (7/24 - 1/2,0) -- (-17/48, -\y) -- (-1/2,0); 
        \node at ({7/48-1/2},0) {\color{Purple}$\phi_{(0,1)}(V)$};
        \filldraw[fill=SpringGreen!25] (7/24 - 1/2,0) -- (-7/48,-\z) -- (0, -\x) -- (-1/16, -\y) -- (7/24 - 1/2,0); 
        \node at (-5/48,-\x/2) {\color{Green}$\phi_{(1,1)}(V)$};
        \filldraw[fill=Blue!25] (0, -\x) -- (7/48,-\z) -- (1/2 - 7/24,0) -- (1/16, -\y) -- (0, -\x); 
        \node at (5/48,-\x/2) {\color{Blue}$\phi_{(2,1)}(V)$};
        \filldraw[fill=magenta!25] (1/2 - 7/24,0) -- (17/48,\y) -- (1/2,0) -- (17/48, -\y) -- (1/2 - 7/24,0); 
        \node at ({-7/48+1/2},0) {\color{magenta}$\phi_{(3,1)}(V)$};
        \filldraw[black] (-1/2,0) circle (0pt) node[anchor=east] {$P_{\!-}$};
        \filldraw[black] (1/2,0) circle (0pt) node[anchor=west] {$P_{\!+}$};
        \filldraw[black] (0,\x) circle (0pt) node[anchor=south] {$Q_{\!+}$};
        \filldraw[black] (0,-\x) circle (0pt) node[anchor=north] {$Q_{\!-}$};

    \end{tikzpicture}}
    \caption{The rhombus $V$ and the images of $V$ under $\phi_{(0,1)}$, $\phi_{(1,1)}$, $\phi_{(2,1)}$ and $\phi_{(3,1)}$ for $p = \tfrac{7}{24}$.}
    \label{Fig:Robus-in}
\end{figure}

\newpage

To construct a generalised $p$-Koch curve one applies the maps $\phi_{(a,b)}$, along certain sequences in $A^{\mathbb N}$, where $A = \{0,1,2,3\}\times \{0,1\}$, much like in \eqref{eq : p-Koch curve}. More precisely, we record the orientation of the triangles whose sides replace the middle ($1-2p$)-section of each line segment via a sequence of tuples $\mathbf s = ((s_k(0),s_k(1),...,s_k(4^{k}-1)))_{k = 0}^\infty$ belonging to $\mathds{S} = \prod_{k=0}^{\infty} \{ 0, 1\}^{4^k}$. For each such $\mathbf{s}$ and each $\mathbf{u} = (u_k)_{k \in \mathbb N} \in \{0,1,2,3\}^{\mathbb{N}}$ we set $\omega_1(\mathbf{u}, \mathbf{s}) = (u_1, s_0(0))$ and, for $k\in \mathbb N_{\geq 2}$, we let
    \begin{align}\label{eq:omega_u_s}
    \textstyle
    \omega_k(\mathbf{u}, \mathbf s) = (u_{k}, s_{k-1}(\sum_{j=1}^{k-1}u_j 4^{k-1-j})) \in A.
    \end{align}
Letting $\omega(\mathbf{u},\mathbf{s}) = (\omega_k(\mathbf{u},\mathbf s))_{k\in \mathbb N}$, we set
\begin{align}\label{eq:Sigma_s}
    \Sigma(\mathbf s) = \left \{ \omega(\mathbf{u},\mathbf s) = (\omega_k(\mathbf{u},\mathbf s))_{k\in \mathbb N} :  \mathbf{u} \in \{0,1,2,3\}^{\mathbb N} \right\} \subset A^{\mathbb N}.
\end{align}
The generalised $p$-Koch curve associated with $\mathbf s$ is then given by the set
\begin{align*} 
    \mathcal C(\mathbf s) = \bigcup_{(\omega_k)_{k\in \mathbb N} \in \Sigma(\mathbf s)}
    \bigcap_{k\in \mathbb N} \left(\phi_{\omega_1}\circ \phi_{\omega_2} \circ \cdots \circ \phi_{\omega_k}\right)(V).
\end{align*}
We show that the $\mathcal C(\mathbf s)$ is again a continuous curve, see \Cref{fig:enter-label2} for illustrations of approximations of two generalised $p$-Koch curves.

\begin{figure}[h!]    \centering{\transparent{0.9}\includegraphics[width=0.49\linewidth]{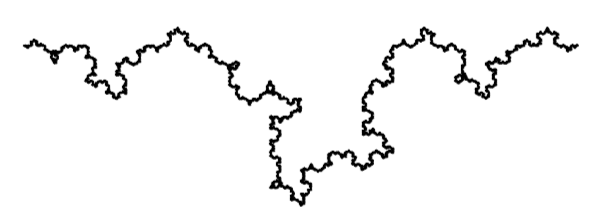}}
    \includegraphics[width=0.49\linewidth]{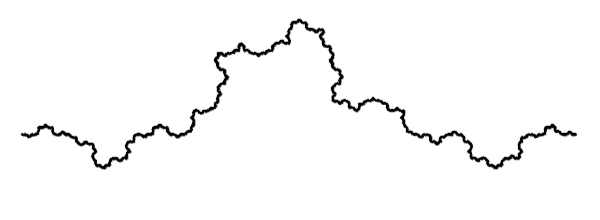}
    \caption{\textls[-16]{Approximations of generalised $p$-Koch curves for $p=\tfrac13$ (left) and $p=(2+\sqrt{2})^{-1}$ (right).}}
    \label{fig:enter-label2}
\end{figure}

A geometric way to construct the curve is to again consider the action of the maps on \mbox{$E_0 = [-\tfrac12,\tfrac12]\times \{0\}$ $\subset \mathbb{R}^{2}$}. For this, we fix an $\mathbf s \in \mathds{S}$ and set $C_{0}(\mathbf{s}) = E_{0}$. If the first entry of $\mathbf s$ is $(0)$ we set $C_{1}(\mathbf{s})$ to be the union of line segments $\phi_{(i,0)}(E_0)$ for $i\in \{0,1,2,3\}$; and if the first entry of $\mathbf s$ is $(1)$ we set $C_{1}(\mathbf{s})$ to be the union of line segments $\phi_{(i,1)}(E_0)$ for $i\in \{0,1,2,3\}$. That is, we replace the original line segment by two outer segments of length $p$ and replace the middle segment of length $1-2p$ by two sides of an isosceles triangle with legs of proportional length $p$ orientated upwards if the first entry of $\mathbf s$ is $(0)$; and downwards if the first entry of $\mathbf s$ is $(1)$. We then repeat this process iteratively on each of the line segments of $C_{1}(\mathbf{s})$ to generate $C_{2}(\mathbf{s})$, and continue the process \textsl{ad infinitum}, see \Cref{fig:nolabel2} for an illustration of this iterative process. Formally, we begin with $C_0(\mathbf s) = E_0$, and for $k\in \mathbb N$ and  $\mathbf{u} = (u_1,u_2,...,u_k)\in \{0,1,2,3\}^{k}$, we let $E^k_{\mathbf{u}}(\mathbf s)$ denote the line segment $\phi_{(u_1,s_0(0))}\circ \phi_{(u_2,s_1(u_1))}\circ\cdots \circ \phi_{(u_{k},s_{k-1}(\sum_{j=1}^{k-1} u_j\,4^{k-1-j}))}(E_0)$ of length $p^k$, and set
\begin{align} \label{eq : approximating curve}
    \mathcal C_k(\mathbf s) = \bigcup_{\mathbf{u} \in \{0,1,2,3\}^{k}} E^k_{\mathbf{u}}(\mathbf s).
\end{align}
For each $k\in \mathbb N$, the set $\mathcal C_k(\mathbf s)$ is then a continuous curve consisting of $4^k$ connected line segments of length $p^k$. Moreover, it can be shown that, as $k$ increases, the sets $\mathcal C_k(\mathbf s)$ converge to $\mathcal C(\mathbf s)$ with respect to the Hausdorff metric.

\tikzmath{ \x = sqrt(3)/2;}
\begin{figure}[ht]
\centering
    \begin{tikzpicture}[scale = 1.2]
    \draw[black, thick, -] (0,0) -- (3,0) node[midway,below]{};
    \draw[black, thick, -] (4,0) -- (5,0) -- (5.5,-\x) -- (6,0) -- (7,0);
    \draw[black, thick, -] (8,0) -- (8 + 1/3,0) -- (8.5,\x/3) -- (8 + 2/3, 0) -- (9,0);
    \draw[black, thick, -] (9,0) -- (9+1/6,-\x/3) -- (9,-\x/3-\x/3) -- (9.5-1/6,-\x + \x/3) -- (9.5,-\x);
    \draw[black, thick, -] (9.5,-\x) -- (9.5+1/6,-\x + \x/3) -- (9.5, -\x/3) -- (10-1/6,-\x/3) -- (10,0);
    \draw[black, thick, -] (10,0) -- (10 + 1/3,0) -- (10.5,-\x/3) -- (10 + 2/3, 0) --  (11,0);    
    \end{tikzpicture}
    \caption{The curves $\mathcal C_0(\mathbf s)$, $\mathcal C_1(\mathbf s)$ and $\mathcal C_2(\mathbf s)$ where $s_0 = (1)$, $s_1 = (0,1,0,1)$ for $p=\tfrac13$.}
    \label{fig:nolabel2}
\end{figure}
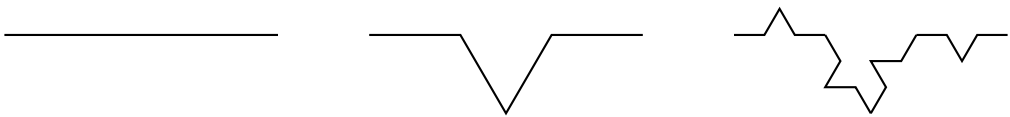

Note, if $s_k(i) = 0$ for all $k\in \mathbb N_0$ and $i \in \{0,1,...,4^{k}-1\}$, then the associated generalised $p$-Koch curve is the $p$-Koch curve in \eqref{eq : p-Koch curve}, whereas taking $s_k(i) = 1$ for every $k\in \mathbb N_0$ and $i \in \{0,1,...,4^{k}-1\}$ results in a $p$-Koch curve reflected in the $x$-axis.

From the $p$-Koch curves one may construct the \textbf{$p$-Koch snowflake}. This snowflake is built by placing three $p$-Koch curves along the sides of an equilateral triangle. Doing this but with three upside-down $p$-Koch curves results in what we refer to as the \textbf{$p$-Koch anti-snowflake}, see \Cref{fig:snowflake and antisnowflake}. 

\begin{figure}[ht]
    \centering
    \includegraphics[width=0.25\linewidth]{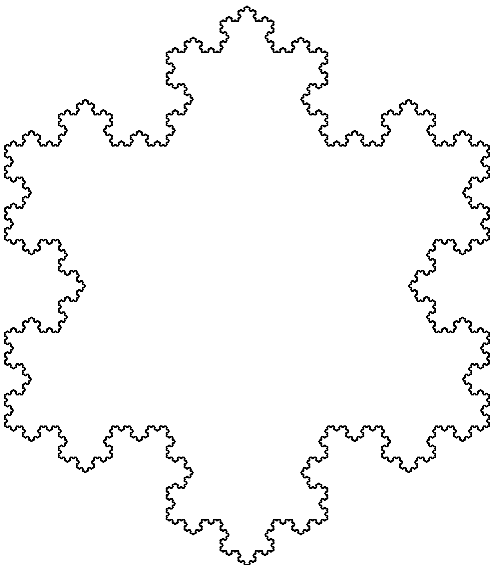}
    \hspace{3.5em} 
    \includegraphics[width=0.25\linewidth, angle=60]{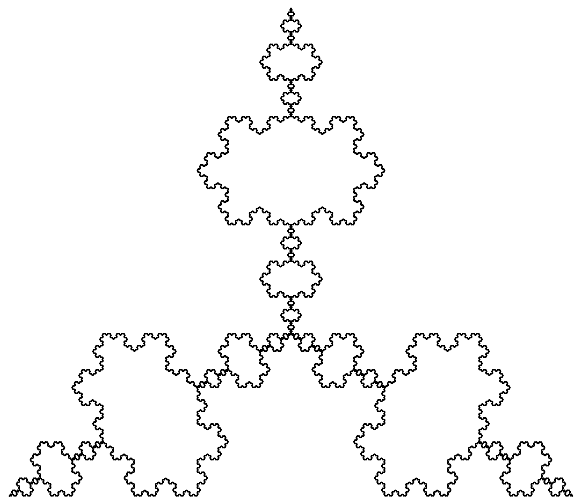}
    \caption{Approximations of the $p$-Koch snowflake (left) and the $p$-Koch anti-snowflake (right) for $p=\tfrac13$.}
    \label{fig:snowflake and antisnowflake}
\end{figure}

In the present article we focus on the geometry of a family of \textbf{generalised $p$-Koch snowflakes}, which are constructed by placing three generalised $p$-Koch curves along the sides of an equilateral triangle. Since a generalised $p$-Koch curve is defined by the corresponding sequence $\mathbf s = ((s_k(0),...,s_k(4^{k}-1)))_{k=0}^\infty$, a generalised $p$-Koch snowflake is defined by three such sequences, $\mathbf s$, $\mathbf t$ and $\mathbf r$. Formally, if we define the two affine maps $f_1, f_2: \mathbb R^2 \to \mathbb R^2$ by
\begin{align}\label{eq:translation_maps}
    f_1\begin{pmatrix}
        x\\
        y
    \end{pmatrix} = -\frac12\begin{pmatrix} 
        1 & -\sqrt{3}\\
        \sqrt{3} & 1
    \end{pmatrix}
    \begin{pmatrix}
        x\\
        y
    \end{pmatrix}
    +
    \frac{1}{4} \begin{pmatrix}
        1\\
        -\sqrt{3}
    \end{pmatrix}
    \quad \text{and} \quad
    f_2\begin{pmatrix}
        x\\
        y
    \end{pmatrix} = -\frac12\begin{pmatrix} 
        1 & \sqrt{3}\\
        -\sqrt{3} & 1
    \end{pmatrix}
    \begin{pmatrix}
        x\\
        y
    \end{pmatrix}
    -
    \frac{1}{4} \begin{pmatrix}
        1\\
        \sqrt{3}
    \end{pmatrix},
\end{align}
the generalised $p$-Koch snowflake corresponding to the choice sequences $\mathbf s$, $\mathbf t$ and $\mathbf r$ is given by
\begin{align*}
    \mathcal{F}(\mathbf s,\mathbf t,\mathbf r) = \mathcal C(\mathbf s) \cup f_1( \mathcal C(\mathbf t)) \cup f_2( \mathcal C(\mathbf r)),
\end{align*}
see \Cref{fig:random-snowflakes} for illustrations of approximations of three such $p$-Koch snowflakes.

\begin{figure}[ht]
    \centering
    \includegraphics[width=0.32\linewidth]{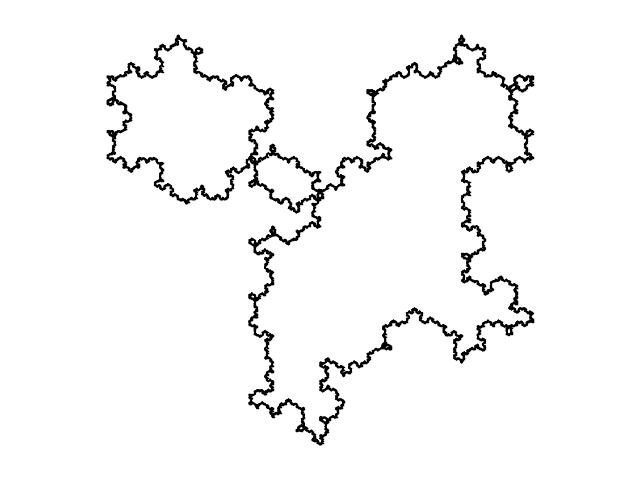}
    \includegraphics[width=0.32\linewidth]{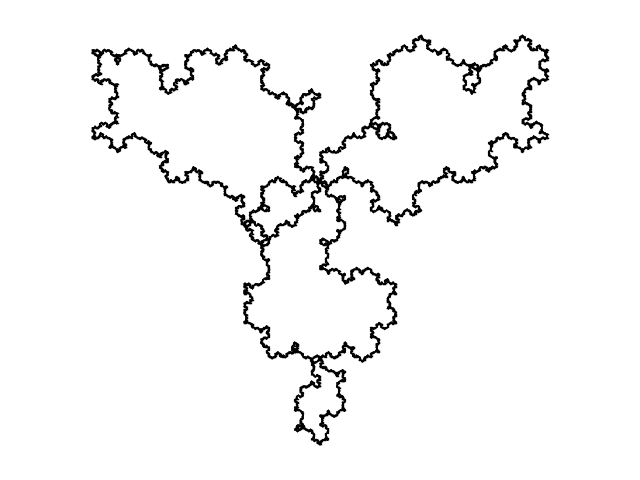}
    \includegraphics[width=0.32\linewidth]{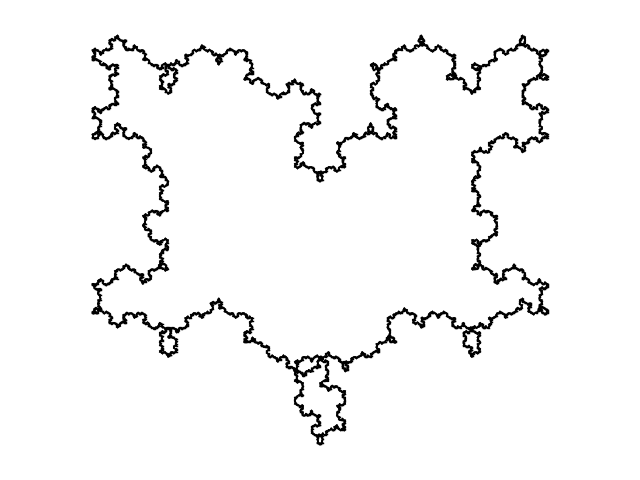}
    \caption{Approximations of three examples of generalised $p$-Koch snowflakes for $p=\tfrac13$.}
    \label{fig:random-snowflakes}
\end{figure}

The construction of the generalised $p$-Koch snowflake has parallels to that of Rohde snowflakes, which are particularly important domains as, up to bi-Lipschitz transformations, every planar quasidisk (namely a domain that is bounded by the image of a circle under a quasi-conformal map) can be realised as a Rohde snowflake \cite{Rhode2001}.

Our aim is to address the following four natural questions about generalised $p$-Koch snowflakes.
\begin{itemize}
    \item[(Q1)] As Hausdorff dimension is a natural example of a fractal dimension, which can be used to quantify the complexity of a fractal set, do the generalised $p$-Koch curves and -snowflakes have the same Hausdorff dimension as the classical $p$-Koch curve and $p$-Koch-snowflake?
    \item[(Q2)] Are generalised $p$-Koch snowflakes Jordan curves and, if so, are they quasicircles?
    \item[(Q3)] Let $y_{p}$ be the area enclosed by the $p$-Koch snowflake and let $x_{p}$ be the area enclosed by the $p$-Koch anti-snowflake. Given $y \in [x_{p},y_{p}]$, does there exist a generalised $p$-Koch snowflake whose enclosed area is $y$?  If so, how many such generalised $p$-Koch snowflakes exist for the given parameter $y$?
    \item[(Q4)] For each point in-between the $p$-Koch snowflake and the $p$-Koch anti-snowflake, does there exist a generalised $p$-Koch snowflake that contains that point?
\end{itemize}
In the sequel we address each of these questions: \Cref{sec22} is devoted to (Q1), in \Cref{sec_quasicircle} we find that the answer to (Q2) is yes in most (but not all) cases, \Cref{sec_area} answers (Q3) positively, and \Cref{sec3} answers (Q4) negatively.  

\section{Well-definedness and fractal properties of generalised \texorpdfstring{$p$}{p}-Koch snowflakes}\label{sec22}

Throughout this section, we let $p\in (\tfrac14,\tfrac13]$ be fixed. We prove that both constructions of the generalised $p$-Koch curves, given in the introduction, are well defined and yield the same fractal sets $\mathcal C(\mathbf s)$. We also prove that each generalised $p$-Koch curve has the same Hausdorff dimension as the classical $p$-Koch curve, answering (Q1). For this we note, for $(a,b)\in \{0,1,2,3\}\times \{0,1\}$, that $\phi_{(a,b)}$ is a similarity with contraction ratio $p$. That is, 
    \begin{align}\label{eq:contration_ratio_of_contractions}
    \lvert \phi_{(a,b)}(\mathbf x) - \phi_{(a,b)}(\mathbf y) \rvert = p \lvert \mathbf{x} - \mathbf{y} \rvert
    \end{align}
for all $\mathbf x, \mathbf y \in \mathbb R^2$, where $\lvert\,\cdot\,\rvert$ denotes the Euclidean distance. As before, we let $V\subset \mathbb R^2$ denote the closed convex hull of $P_{+}$, $P_{-}$, $Q_{+}$ and $Q_{-}$ and note that $V$ is the filled-in closed rhombus with vertices at $P_{\pm} = (\pm \tfrac12,0)^{\top}$ and $Q_{\pm} = (0,\pm \tfrac12\!\sqrt{4p-1})^{\top}$.  Algebraically, the set $V$ can be written as $V = \{(x,y)^{\top}\in \mathbb R^2 :  \lvert 2x \rvert + \lvert 2y(\sqrt{4p-1})^{-1} \rvert \leq 1\}$. Moreover, let $U = \{(x,y)^{\top}\in \mathbb R^2 :  \lvert 2x \rvert + \lvert 2y(\sqrt{4p-1})^{-1} \rvert < 1\}$ denote the interior of $V$.

\begin{lemma} \label{lemma : well-defined and OSC}
    For $(a,b)\in \{0,1,2,3\}\times \{0,1\}$, we have $\phi_{(a,b)}(V) \subset V$ and $\phi_{(a,b)}(U) \subset U$. Further, for $b\in \{0,1\}$ and $a, a' \in \{0,1,2,3\}$ with $a\neq a'$, we have $\phi_{(a,b)}(U) \cap \phi_{(a',b)}(U) = \emptyset$.
\end{lemma}

\begin{proof}
    Since each map $\phi_{(a,b)}$ for $(a,b)\in \{0,1,2,3\}\times \{0,1\}$ is a non-degenerate similarity, the image $\phi_{(a,b)}(V)$ is again a closed rhombus with vertices $\phi_{(a,b)}(P_{\pm})$ and $\phi_{(a,b)}(Q_{\pm})$, while $\phi_{(a,b)}(U)$ is the (non-trivial) interior of this rhombus. We can therefore understand these images by understanding where the vertices of $V$ are mapped to. Recall that $\phi_{(0,0)} = \phi_{(0,1)}$ and $\phi_{(3,0)} = \phi_{(3,1)}$, and observe the following. 
    \begin{alignat}{4}
        \phi_{(0,0)}(P_{\!-}) &= P_{\!-}\qquad
        &\phi_{(0,0)}(P_{\!+}) &= \phi_{(1,0)}(P_{\!-}) = \phi_{(1,1)}(P_{\!-}) = (p-\tfrac12,0)^\top\nonumber\\
        \phi_{(1,0)}(P_{\!+}) &= \phi_{(2,0)}(P_{\!-}) = Q_{\!+}\qquad
        &\phi_{(1,1)}(P_{\!+}) &= \phi_{(2,1)}(P_{\!-}) = Q_{\!-} \nonumber \\ 
        \phi_{(3,0)}(P_{\!+}) &= P_{\!+}\qquad
        &\phi_{(2,0)}(P_{\!+}) &= \phi_{(2,1)}(P_{\!+}) = \phi_{(3,0)}(P_{\!-}) = (\tfrac12-p,0)^\top \nonumber\\
        \phi_{(0,0)}(Q_{\pm}) &= (\tfrac{p-1}2, \pm \tfrac{p}2 \sqrt{4p-1})^\top\qquad
        &\phi_{(1,0)}(Q_{\!+}) &= (-\tfrac{p}2, \tfrac{1-p}2\sqrt{4p-1})^\top \nonumber \\
        \phi_{(1,0)}(Q_{\!-}) &= (\tfrac{3p-1}{2}, \tfrac{p}2\sqrt{4p-1})^\top\qquad
        &\phi_{(1,1)}(Q_{\!+}) &= (\tfrac{3p-1}{2}, -\tfrac{p}2\sqrt{4p-1})^\top \label{eq:end_pts_identicification}\\
        \phi_{(1,1)}(Q_{\!-}) &= (-\tfrac{p}2, -\tfrac{1-p}2\sqrt{4p-1})^\top\qquad
        &\phi_{(2,0)}(Q_{\!+}) &= (\tfrac{p}2, \tfrac{1-p}2\sqrt{4p-1})^\top \nonumber \\ 
        \phi_{(2,0)}(Q_{\!-}) &= (\tfrac{1-3p}{2}, \tfrac{p}2\sqrt{4p-1})^\top\qquad
        &\phi_{(2,1)}(Q_{\!+}) &= (\tfrac{1-3p}{2}, -\tfrac{p}2\sqrt{4p-1})^\top \nonumber\\
        \phi_{(2,1)}(Q_{\!-}) &= (\tfrac{p}2, -\tfrac{1-p}2\sqrt{4p-1})^\top\qquad
        &\phi_{(3,0)}(Q_{\pm}) &= (\tfrac{1-p}2, \pm \tfrac{p}2 \sqrt{4p-1})^\top \nonumber
    \intertext{Elementary calculations show that each of the above points satisfies the defining relation $\lvert 2x \rvert + \lvert 2y(\sqrt{4p-1})^{-1} \rvert \leq 1$ of $V$, and thus $\phi_{(a,b)}(P_\pm),\phi_{(a,b)}(Q_\pm)\in V$. By this, the linearity of our maps, and the convexity of $V$, it follows that $\phi_{(a,b)}(V) \subset V$ and $\phi_{(a,b)}(U)\subset U$ for all $(a,b)\in \{0,1,2,3\}\times \{0,1\}$. For the mutual disjointness of the sets $\phi_{(a,b)}(U)$ we notice the following.}
        \phi_{(0,0)}(U) &= \phi_{(0,1)}(U) \subset (-\tfrac12,p-\tfrac12) \times \mathbb R\qquad 
        &\phi_{(1,0)}(U) &\subset (p-\tfrac12,0)\times (0,\tfrac12\sqrt{4p-1}) \nonumber \\
        \phi_{(1,1)}(U) &\subset (p-\tfrac12,0)\times (-\tfrac12\sqrt{4p-1}, 0)\qquad
        &\phi_{(2,0)}(U) &\subset (0,\tfrac12-p)\times(0,\tfrac12\sqrt{4p-1}) \nonumber \\
        \phi_{(2,1)}(U) &\subset (0,\tfrac12-p)\times (-\tfrac12\sqrt{4p-1},0) \qquad
        &\phi_{(3,0)}(U) &= \phi_{(3,1)}(U) \subset (\tfrac12-p,\tfrac12)\times \mathbb R \nonumber
    \end{alignat}
    Since these sets are mutually disjoint, so are the sets $\phi_{(a,b)}(U)$.
\end{proof}

The first part of \Cref{lemma : well-defined and OSC} together with \eqref{eq:contration_ratio_of_contractions}, shows that the maps $\phi_{(a,b)}$ for $(a,b)\in \{0,1,2,3\}\times \{0,1\}$ are contractive self-maps on the compact set $V$. Therefore, for any sequence $\omega = (\omega_k)_{k\in \mathbb N} \in A^{\mathbb N}$, where $A$ is the alphabet $\{0,1,2,3\}\times \{0,1\}$, it follows that $((\phi_{\omega_1}\circ \phi_{\omega_2}\circ \cdots \circ \phi_{\omega_k})(V))_{k\in \mathbb N}$ is a nested sequence of compact sets with decreasing diameter, and so, by the Cantor Intersection Theorem, the intersection $\bigcap_{k\in \mathbb N} (\phi_{\omega_1}\circ \phi_{\omega_2}\circ \cdots \circ \phi_{\omega_k})(V)$ contains precisely one element, that is $\pi(\omega) = \lim_{k\rightarrow \infty} \phi_{\omega_1}\circ \phi_{\omega_2}\circ \cdots \circ \phi_{\omega_k}(x_0)$, where this limit is independent of the choice of $x_0\in V$. Hence, if we define the \textbf{projection map} $\pi: A^{\mathbb N} \rightarrow V$ by $\omega \mapsto \pi(\omega)$, then we see that $\mathcal C(\mathbf s) = \pi(\Sigma(\mathbf s))$ is a well-defined and non-empty subset of $V$ for each sequence $\mathbf s = ((s_k(0),s_k(1),...,s_k(4^{k}-1)))_{k = 0}^\infty $ with $s_k(i) \in \{0,1\}$ for every $k\in \mathbb N_0$ and $i\in \{0,1,...,4^{k}-1\}$.

In the cases where there is some $s\in \{0,1\}$ such that $\mathbf s$ is given by $s_k(i) = s$ for each $k\in \mathbb N_0$ and $i\in \{0,1,...,4^k-1\}$, we have $\Sigma(\mathbf s) = (\{0,1,2,3\}\times \{s\})^{\mathbb N}$. The case $s=0$ describes the $p$-Koch curve and the case $s=1$ describes an upside-down Koch curve. In these situations it is well-known that the self-map $\Phi_s$ on the space of compact non-empty subsets of $V$ given by $\Phi_s(E) = \bigcup_{i=0}^3 \phi_{(i,s)}(E)$ is a contraction with respect to the Hausdorff metric. By the Banach fixed-point theorem there exists a unique non-empty compact set $C\subseteq V$ satisfying $\Phi_s(C) = C$. In fact, we have $C = \mathcal C(\mathbf s)$ and say that $\mathcal C(\mathbf s)$ is the attractor set of the iterated function system $\{\phi_{(i,s)} : i\in \{0,1,2,3\}\}$ and that $\Phi_s$ is the corresponding \textbf{Hutchinson operator}, see \cite{hutchinson1981fractals}. 

By similar arguments to those above, for any non-empty subset $E \subseteq V$ and any sequence $\omega = (\omega_k)_{k\in \mathbb N} \in A^{\mathbb N}$ the sequence of sets $((\phi_{\omega_1}\circ \cdots \circ \phi_{\omega_k})(E))_{k\in \mathbb N}$ converges in the Hausdorff metric to the singleton $\{\pi(\omega)\}$. In particular, taking $E = E_0 = [-\tfrac12,\tfrac12]\times\{0\} \subset V$, we see, for any $\mathbf{s} \in \prod_{k=0}^{\infty} \{ 0, 1\}^{4^k}$, that the curves $\mathcal C_k(\mathbf s)$ defined in \eqref{eq : approximating curve} converge to the generalised $p$-Koch curve $\mathcal C(\mathbf s)$ in the Hausdorff metric. Further, observe that, for $(a, b) \in \{0, 1, 2, 3\} \times \{0,1\}$, the set $\phi_{(a,b)}(E_0)$ is a line segment of length $p$ with endpoints $\phi_{(a,b)}(P_{\pm})$. Since we have, for $b\in \{0,1\}$, that $\phi_{(0,b)}(P_{\!-}) = P_{\!-}$, $\phi_{(j+1,b)}(P_{\!-}) = \phi_{(j,b)}(P_{\!+})$ for each $j\in \{0,1,2\}$, and $\phi_{(3,b)}(P_{\!+}) = P_{\!+}$, it follows that, for $k\in \mathbb N_0$, each set $\mathcal C_k(\mathbf s)$ is a continuous curve with endpoints $P_\pm$ consisting of $4^k$ line segments of length $p^k$. These observations in tandem with \Cref{lemma : well-defined and OSC}, yield that the generalised $p$-Koch curve $\mathcal C(\mathbf s)$ is a continuous curve of infinite length with endpoints $P_\pm$.

The second part of \Cref{lemma : well-defined and OSC} shows that $\{\phi_{(i,0)} : i\in \{0,1,2,3\}\}$ and $\{\phi_{(i,1)} : i\in \{0,1,2,3\}\}$ satisfy the \textbf{open set condition}. That is, for $s\in \{0,1\}$ there exists a non-empty open set $U\subset V$ such that $\bigcup_{i=0}^3 \phi_{(i,s)}(U) \subseteq U$ and the sets $\phi_{(i,s)}(U)$, for $i\in \{0,1,2,3\}$, are pairwise disjoint. It was shown in \cite[Theorem 5.3(1)]{hutchinson1981fractals} that the Hausdorff dimension of the attractor of an iterated function system $\{ g_i : i\in I\}$ consisting of finitely many similarities and satisfying the open set condition is given by the unique number $r$ solving $\sum_{i\in I} c_i^r = 1$ where $c_i$ is the contraction ratio of $g_i$. Since each map $\phi_{(a,b)}$ is a similarity of contraction ratio $p$, this implies that the $p$-Koch curve and the upside-down $p$-Koch curve have Hausdorff dimension $\log(4)/\log(p^{-1})$. In the following we show that all generalised $p$-Koch curves have the same Hausdorff dimension, thus giving an affirmative answer to (Q1), using the same steps as in the proof for \cite[Theorem 9.3]{falconer2013fractal}.

\begin{prop}\label{prop:dimension}
    For $p\in (\tfrac14,\tfrac13]$ and $\mathbf{s} \in \mathds{S}$, the Hausdorff dimension of $\mathcal C(\mathbf{s})$ is
    \begin{align*}
        \dim_{\mathcal{H}}(\mathcal C(\mathbf s)) = \frac{\log(4)}{\log(p^{-1})}.
    \end{align*}
\end{prop}

Since any generalised $p$-Koch snowflake $\mathcal F(\mathbf s, \mathbf t, \mathbf r)$ is given by $\mathcal C(\mathbf s)\cup f_1(\mathcal C(\mathbf t)) \cup f_2(\mathcal C(\mathbf r))$ for some $\mathbf s, \mathbf t,\mathbf r \in \mathds{S}$, since $f_1$ and $f_2$ are both bi-Lipschitz, and since bi-Lipschitz mappings preserve Hausdorff dimension, \Cref{prop:dimension} implies that $\dim_{\mathcal{H}}(\mathcal F(\mathbf s, \mathbf t, \mathbf{r})) = \max\{\dim_{\mathcal{H}}(C(\mathbf q)) : \mathbf q\in\{\mathbf{s}, \mathbf t, \mathbf r\}\} = \log(4) / \log(p^{-1})$.

\begin{proof}[Proof of \Cref{prop:dimension}]
    Let $p \in (\tfrac14, \tfrac13]$ and $\mathbf s = ((s_k(0),...,s_k(4^{k}-1)))_{k=0}^\infty\in \mathds{S}$ be fixed, and set $d = \log(4)/\log(p^{-1})$.
    For $k\in \mathbb N$, $u_1,...,u_k\in \{0,1,2,3\}$ and $B\subseteq V$, set 
    \begin{align*}
        B^{u_1,...,u_k} = \phi_{(u_1,s_0(0))}\circ \phi_{(u_2,s_1(u_1))}\circ\cdots \circ \phi_{(u_{k},s_{k-1}(\sum_{j=1}^{k-1} u_j 4^{k-1-j}))}(B).
    \end{align*}
    In particular, for $k\in \mathbb N$, we have $\mathcal C(\mathbf s) \subset \bigcup_{u_1,...,u_k\in \{0,1,2,3\}} V^{u_1,...,u_k}$. For each $\omega_1,...,\omega_k \in \{0,1,2,3\}\times\{0,1\}$ the composition $\phi_{\omega_1}\circ...\circ \phi_{\omega_k}$ is a similarity with contraction ratio $p^k$, so each set $V^{u_1,...,u_k}$ is a closed rhombus of diameter $p^k$. For $r, \delta \in \mathbb{R}$, with $r \geq 0$ and $\delta > 0$, let $H^r$ denote the $r$-dimensional Hausdorff measure and $H^r_{\delta}$ its $\delta$-approximate, see \cite{falconer2013fractal} for precise definitions. 
    For each $r>d$ we have $4p^r < 1$ and so 
    \begin{align*}
         H^r(\mathcal C(\mathbf s)) = \lim_{k\rightarrow \infty} H^r_{\! p^{k}}(\mathcal C(\mathbf s)) \leq \lim_{k\rightarrow \infty } \sum_{u_1,...,u_k\in \{0,1,2,3\}} \diam(V^{u_1,...,u_k})^r = \lim_{k\rightarrow \infty} 4^k \cdot p^{kr} = \lim_{k\rightarrow \infty} (4p^r)^k = 0.
    \end{align*}
    Hence, $\dim_{\mathcal{H}}(\mathcal C(\mathbf s)) \leq d$. For the opposite inequality, we define a finite measure on $\mathbb R^2$ that has support $\mathcal C(\mathbf s)$ and prove it satisfies the conditions of the \textbf{mass distribution principle}, see \cite[Section 4.1]{falconer2013fractal}, allowing us to conclude that $d \leq \dim_{\mathcal{H}}(\mathcal C(\mathbf s))$. 
    
    Thus, it remains to show that there exists a finite measure on $\mathbb R^2$ that has support $\mathcal C(\mathbf s)$ satisfying the conditions of the mass distribution principle, namely, that there exists a constant $c$ and a Borel measure $\mu$ such that, given a Borel set $M$ of diameter $r < 1$,  the $\mu$-measure of $M$ is bounded above by $c \cdot r^{d}$.

    To this end, let $I = \{0,1,2,3\}^{\mathbb N}$ and for each $k\in \mathbb N$ and $u_1,...,u_k\in \{0,1,2,3\}$ we consider the \textbf{cylinder set} $I_{u_1,...,u_k} = \left \{(i_m)_{m\in \mathbb N} \in I : i_1 = u_1,..., i_k = u_k \right\}$; when equipping $\{0,1,2,3\}$ with the discrete topology and $I$ with the product topology, the set of cylinder sets form a basis for this topology on $I$. We define a Borel measure $\mu$ on $I$ by setting $\mu(I_{u_1,...,u_k}) = 4^{-k}$ for each $k\in \mathbb N$ and each $u_1,...,u_k\in \{0,1,2,3\}$. Note, by Carath\'eodory's extension theorem, that $\mu$ is a well-defined measure and that $\mu(I) = \sum_{i=0}^3 \mu(I_i) = 4\cdot 4^{-1} = 1$. For $\mathbf{u} \in I$ we write $x_{\mathbf{u}} = \pi(\omega(\mathbf{u},\mathbf{s}))$, and recall that $\mathcal C(\mathbf s) = \pi(\Sigma(\mathbf s)) = \{ x_{\mathbf{u}} : \mathbf{u} \in I \}$, where $\omega(\mathbf{u},\mathbf{s})$ is as defined in \eqref{eq:omega_u_s}. We define a measure $\widetilde \mu$ on $\mathbb R^2$ by setting $\widetilde \mu(M) = \mu(\{\mathbf{u} \in I : x_{\mathbf u} \in M\})$ for any Borel set $M\subseteq \mathbb R^2$. Note that the set $\{\mathbf{u}\in I : x_{\mathbf u} \in M\}$ is measurable since, by construction, $\pi(\omega(\cdot,\mathbf{s}))$ is continuous, and observe $\widetilde \mu(M) = \widetilde \mu(M\cap \mathcal C(\mathbf s))$ for each Borel set $M\subseteq \mathbb R^2$ and that $\widetilde \mu(\mathbb R^2) = \widetilde \mu(\mathcal C(\mathbf s)) = \mu(I) = 1$. We will verify the conditions of the mass distribution principle, namely that there exists a constant $c>0$ satisfying $\widetilde \mu(M) \leq c\cdot\diam(M)^d$ for each Borel set $M\subset \mathbb R^2$ satisfying $\diam(M) < 1$.

    Let $M \subset \mathbb R^2$ denote a fixed Borel set with $\diam(M) < 1$. By definition, there exists some ball $B$ of radius $r = \diam(M)<1$ that contains $M$. Let $m\in \mathbb N$ be the unique integer satisfying $p^m \leq r < p^{m-1}$. By \Cref{lemma : well-defined and OSC} the collection $\{U^{u_1,...,u_m} : u_1,...,u_m\in \{0,1,2,3\}\}$ is disjoint, where we recall that $U$ is the interior of $V$. Furthermore, we have for each $u_1,...,u_m\in \{0,1,2,3\}$ that $V^{u_1,...,u_m} = \text{cl}(U)^{u_1,...,u_m} =  \text{cl}(U^{u_1,...,u_m})$, where $\text{cl}(\cdot)$ denotes taking closure with respect to the Euclidean norm $\lvert \,\cdot\, \rvert$.
    
    Note that $U$ is contained in a ball of radius $a_1 = \tfrac12$, while it contains a ball of radius $a_2 = \tfrac14\!\sqrt{4p-1}$. Hence, for each $u_1,...,u_m\in \{0,1,2,3\}$, the open rhombus $U^{u_1,...,u_m}$ is contained in a ball of radius $p^m a_1\leq ra_1$ and contains a ball of radius $p^ma_2>pra_2$. If we define the set $J = \{(u_1,...,u_m)\in \{0,1,2,3\}^m : B\cap V^{u_1,...,u_m} \neq \emptyset\}$, then it follows from \cite[Lemma 9.2]{falconer2013fractal} that $\#J$ is bounded by  
    $c = (1+2a_1)^2(p a_2)^{-2} > 0$. It follows that 
    \begin{align*}
        \widetilde \mu(M) 
        \leq \widetilde \mu(B\cap \mathcal C(\mathbf s)) 
        \leq \sum_{(u_1,...,u_m)\in J} \mu(I_{u_1,...,u_m}) 
        = \#J \cdot 4^{-m} 
        \leq cp^{m d} 
        \leq r^{d}c 
        = c\cdot \diam(M)^d.
    \end{align*}
    By the Mass Distribution Principle we conclude that $d \leq \dim_{\mathcal{H}}(\mathcal C(\mathbf s))$ and hence that $\dim_{\mathcal{H}}(\mathcal C(\mathbf s)) = d$.
\end{proof}

\section{Generalised \texorpdfstring{$p$}{p}-Koch snowflakes as quasicircles} \label{sec_quasicircle}

In this section we study conditions under which generalised $p$-Koch snowflakes are quasicircles, answering (Q2). A \textbf{Jordan curve} is a planar set homeomorphic to the unit circle $\mathbb{T} = \{ (x, y) \in \mathbb{R}^{2} : x^{2}+y^{2} = 1 \}$, and a \textbf{quasicircle} is a Jordan curve $C$ with \textbf{bounded turning}, meaning there exists a constant $K\in [1,\infty)$ such that for every $x,y\in C$ there exists a compact and connected subset $\Gamma \subseteq C$, also referred to as a \textbf{continuum}, containing both $x$ and $y$, and such that $\diam(\Gamma) \leq K\cdot \lvert x-y \rvert$.

We therefore begin by showing that in most cases, generalised $p$-Koch snowflakes are Jordan curves. To this end, we let $p\in (\tfrac14,\tfrac13]$ and we construct a natural parameterisation of any given generalised $p$-Koch snowflake. We begin by parameterising the generalised $p$-Koch curves. Let $\mathbf s \in \mathds S$ be fixed. For $t\in [0,1]$ we choose a sequence $\mathbf u_t = (u_k)_{k\in \mathbb N} \in \{0,1,2,3\}^{\mathbb N}$ satisfying $t = \sum_{k\in \mathbb N} u_k 4^{-k}$. That is, $\mathbf u_t$ consists of the digits of a base 4 number expansion for $t$. We then define $\rho_{\mathbf s}(t)$ to be $\pi(\omega(\mathbf u_t, \mathbf s))$. That is, if we write $\omega(\mathbf u_t,\mathbf s) = (\omega_k)_{k\in \mathbb N}$, then $\rho_{\mathbf s}(t)$ is the sole element of the set
\begin{align*}
	\bigcap_{k \in \mathbb{N}} \phi_{\omega_{1}} \circ \phi_{\omega_{2}} \circ \cdots \circ \phi_{\omega_{k}}(V).
\end{align*}
Using the fact that, for $z \in \{0,1\}^{\mathbb{N}}$,
    \begin{align}\label{eq:Lemma_3_1_intersection}
        \bigcap_{k \in \mathbb{N}} \phi_{(0, z_{1})} \circ \phi_{(0, z_{2})} \circ \cdots \circ \phi_{(0, z_{k})}(V) = \{ P_{-}\}
        \quad \text{and} \quad
        \bigcap_{k \in \mathbb{N}} \phi_{(3, z_{1})} \circ \phi_{(3, z_{2})} \circ \cdots \circ \phi_{(3, z_{k})}(V) = \{ P_{+}\},
    \end{align}
in tandem with \eqref{eq:end_pts_identicification}, an elementary calculation shows that if $\mathbf u = (u_k)_{k\in \mathbb N}$ and $\mathbf v = (v_k)_{k\in \mathbb N}$ both yield base 4 number expansions for $t$, that is $t = \sum_{k\in \mathbb N} u_k 4^{-k} = \sum_{k\in \mathbb N} v_k 4^{-k}$, then $\pi(\omega(\mathbf u,\mathbf s)) = \pi(\omega(\mathbf v, \mathbf s))$. This means that the map $\rho_{\mathbf s}:[0,1]\to \mathcal C(\mathbf s)$ is well-defined and independent on the choice of $\mathbf u_t$ for any given $t\in [0,1]$. 

\begin{lemma}\label{lemma : properties rho_s}
    Let $p\in (\tfrac14, \tfrac13]$ and let $\mathbf s\in \mathds S$ be fixed. The map $\rho_{\mathbf s}:[0,1]\to \mathcal C(\mathbf s)$ is continuous and surjective. Moreover, if $p\in (\tfrac14,\tfrac13)$, then $\rho_{\mathbf s}$ is injective. 
\end{lemma}

\begin{proof} 
Surjectivity follows from the fact that each sequence in $\{0,1,2,3\}^{\mathbf N}$ corresponds to the digits of a base~4 expansion of a number in $[0,1]$. For continuity, fix $t\in [0,1]$ and $\varepsilon > 0$. Let $m \in \mathbb{N}$ be such that $2 p^{m} < \varepsilon$, and suppose that $t' \in [0,1]$ satisfies $\lvert t - t' \rvert < 4^{-(m+1)}$.  Then there exists $k \in \{ 0, \dots, 4^{m}-1\}$ with either $t, t' \in [k \cdot 4^{-m}, (k+1) \cdot 4^{-m}]$, or $\max\{ t, t'\} \in [k \cdot 4^{-m}, (k+1) \cdot 4^{-m}]$ and $\min\{ t, t'\} \in [(k - 1) \cdot 4^{-m}, k \cdot 4^{-m}]$. Let $(u_{1}, \dots, u_{m})$ and $(v_{1}, \dots, v_{m})$ be the unique elements of $\{ 0,1,2,3\}^{m}$ with $k = \sum_{i = 0}^{m-1} u_{i} \cdot 4^{m-1-i}$ and $(k -1) = \sum_{i = 0}^{m-1} v_{i} \cdot 4^{m-1-i}$. Set $\omega_{1} = (u_{1}, s_{0}(0))$ and $\nu_{1} = (v_{1}, s_{0}(0))$, and for $j \in \{2, \dots, m\}$, set $\omega_{j} = (u_{j}, s_{j-1}(\sum_{i=1}^{j-1} u_{j} 4^{j-1-i}))$ and $\nu_{j} = (v_{j}, s_{j-1}(\sum_{i=1}^{j-1} v_{j} 4^{j-1-i}))$. By construction, we have that $\rho_{\mathbf s}(t)$ and $\rho_{\mathbf s}(t')$ belong to the set $(\phi_{\omega_1} \circ \cdots \circ \phi_{\omega_{m}}(V)) \cup (\phi_{\nu_1} \circ \cdots \circ \phi_{\nu_{m}}(V))$, and utilising \eqref{eq:Lemma_3_1_intersection}, that $\text{diam}((\phi_{\omega_1} \circ \cdots \circ \phi_{\omega_{m}}(V)) \cup (\phi_{\nu_1} \circ \cdots \circ \phi_{\nu_{m}}(V))) < 2 p^{m}$. This yields that $\lvert \rho_{\mathbf s}(t)-\rho_{\mathbf s}(t') \rvert \leq 2 p^{m} < \varepsilon$, and thus that $\rho_{\mathbf s}$ is continuous.

For the final statement, assume that $p\in (\tfrac14,\tfrac13)$. Suppose $t_1,t_2\in [0,1]$ with $t_{1} < t_{2}$ and $\rho_{\mathbf s}(t_1) = \rho_{\mathbf s}(t_2)$. Let $\mathbf u = (u_k)_{k\in \mathbb N}, \mathbf v = (v_k)_{k\in \mathbb N} \in \{0,1,2,3\}^{\mathbb N}$ be such that $t_1 = \sum_{k\in \mathbb N} u_k4^{-k}$ and $t_2 = \sum_{k\in \mathbb N} v_k 4^{-k}$. If $\mathbf u = \mathbf v$, or if there exists $m \in \mathbb{N}$ with $(u_{1}, \dots, u_{m-1}) = (v_{1}, \dots, v_{m-1})$, $u_{m}=v_{m} - 1$ and $u_{j} = 3 = v_{j}+3$ for all integers $j > m$, then $t_1 = t_2$.  Therefore, we may assume that neither of these two cases occur.  Further, observe that if $\eta_k = \omega_k(\mathbf u, \mathbf s)$ and $\xi_k = \omega_k(\mathbf{v}, \mathbf{s})$, then $\eta_k = \xi_k$ for all natural numbers $k\leq m-1$, and $\eta_m \neq \xi_m$, where $m = \min\{k\in \mathbb N : u_k \neq v_k\}$. By construction, $\rho_{\mathbf s}(t_1) \in \phi_{\eta_1} \circ \cdots \circ \phi_{\eta_{m-1}} \circ \phi_{\eta_m}(V)$ while $\rho_{\mathbf s}(t_2) \in \phi_{\eta_1} \circ \cdots \circ \phi_{\eta_{m-1}} \circ \phi_{\xi_m}(V)$. However, since we have $p\in (\tfrac13,\tfrac14)$, the intersection $\phi_{\eta_m}(V)\cap \phi_{\xi_m}(V)$ is only non-empty if $v_m = u_m - 1$ or $v_m = u_m + 1$. Since we have assumed that $t_{1} < t_{2}$, we have the latter of these two situations. Moreover, in this case we have $\phi_{\eta_m}(V)\cap \phi_{\xi_m}(V) = \{\phi_{\eta_m}(P_+)\} = \{\phi_{\xi_m}(P_{-})\}$. Hence $\rho_{\mathbf s}(t_1)=\rho_{\mathbf{s}}(t_2)$ implies for any arbitrary $x_0 \in V$ that $\lim_{k\to \infty} \phi_{\eta_{m+k}}(x_0) = P_{+}$, which implies that $u_k = 3$ for all $k\geq m+1$, while $\lim_{k\to \infty} \phi_{\xi_{m+k}}(x_0) = P_{-}$, which implies that $v_k = 0$ for each $k\geq m+1$. However, at the start of our proof of injectivity, we assumed this not to be the case.
\end{proof}

Note, in the case when $\mathbf{s} = ((s_k(0),s_k(1),...,s_k(4^{k}-1)))_{k = 0}^\infty \in \mathds S$ is such that $s_k(j) = 0$ for all $j \in \{ 0, \dots,  4^{k}-1\}$ and $k \in \mathbb{N}_{0}$, and $p = 1/3$, a similar argument will show that $\rho_{\mathbf s}$ is injective.  Further, observe that as $t$ increases from $0$~to~$1$ the parameterisation $\rho_{\mathbf s}$ traverses the generalised $p$-Koch curve $\mathcal C(\mathbf s)$ from $P_-$ to $P_+$.

Recall that the generalised $p$-Koch snowflakes are defined as the sets $\mathcal F(\mathbf s,\mathbf t, \mathbf r) = \mathcal C(\mathbf s)\cup f_1(\mathcal C(\mathbf t)) \cup f_2(\mathcal C(\mathbf r))$ for $\mathbf s, \mathbf t, \mathbf r \in \mathds S$, with $f_1$ and $f_2$ as defined in \eqref{eq:translation_maps}. A natural choice of parameterisation of the snowflake $\mathcal F = \mathcal F(\mathbf s, \mathbf t, \mathbf r)$ is therefore given by the map $\gamma: [0,1] \to \mathcal F$ defined by
\begin{align} \label{eq : parameterisation}
    \gamma(t) =
            \begin{cases}
            \rho_{\mathbf{s}}(1-3t) & \text{if} \, t \in [0,\tfrac13),\\
            f_{2}(\rho_{\mathbf{r}}(2-3t)) & \text{if} \, t \in [\tfrac13, \tfrac23),\\
            f_{1}(\rho_{\mathbf{t}}(3-3t)) & \text{if} \, t \in [\tfrac23, 1].
            \end{cases}
    \end{align}

    \begin{prop}\label{prop : Jordan curve}
    Fix $p\in (\tfrac14,\tfrac13]$ and $\mathbf s, \mathbf r, \mathbf t \in \mathds S$ and let $\mathcal F = \mathcal F(\mathbf s, \mathbf t, \mathbf r)$ be the corresponding generalised $p$-Koch snowflake. If the restriction of the parameterisation $\gamma:[0,1]\to \mathcal F$ to $[0,1)$ is injective, then $\mathcal F$ is a Jordan curve. In particular, this is the case whenever $p \in (\tfrac14,\tfrac13)$. 
    \end{prop}
    
    \begin{proof}
     We have $\gamma(\tfrac13) = f_2(\rho_{\mathbf r}(1)) = f_2(P_+) = P_- = \rho_{\mathbf s}(0) = \rho_{\mathbf s}(1-3\cdot \tfrac13)$ and $\gamma(\tfrac23) = f_1(P_+) = f_2(P_-) = f_2(\rho_{\mathbf r}(2-3\cdot \tfrac23))$. Since $f_1$ and $f_2$ are continuous, it therefore follows from \Cref{lemma : properties rho_s} that $\gamma$ is continuous and surjective.   
     Furthermore, we have $\gamma(0) = P_+ = \gamma(1)$. By the compactness of the unit circle, it follows that if $\gamma$ is injective on $[0,1)$, then $\mathcal{F}$ is a Jordan curve.

     Fixing $p\in (\tfrac14,\tfrac13)$, by \Cref{lemma : properties rho_s}, the restrictions of $\gamma$ to the intervals $[0,\tfrac13)$, $[\tfrac13,\tfrac23)$ and $[\tfrac23,1)$ are each injective. Since in this case we also have that $V\cap f_1(V) = \{P_+\}$, $f_1(V)\cap f_2(V) = \{f_1(P_+)\} = \{f_{2}(P_-)\}$ and $f_2(V)\cap V = \{P_-\}$, we see that $\gamma$ is also injective on $[0,1)$, and so, in this case, $\mathcal{F}$ is a Jordan curve.
    \end{proof}

    The parameterisation $\gamma$ of $\mathcal F$ is now chosen such that, going from $t=0$ to $t=1$, one traverses $\mathcal F$ in a counter-clockwise direction, starting and ending in $P_+$. 
    Note that if $p = \tfrac13$, there exist cases where $\gamma$ is injective, and hence $\mathcal F$ is a Jordan curve, and cases where it is not. For an example of the former, consider the classical Koch snowflake, whose parameterisation is injective, and for the latter, consider the Koch anti-snowflake, which satisfies $\gamma(\tfrac16) = \gamma(\tfrac12) = \gamma(\tfrac56) = (0, - \tfrac{1}{4} \sqrt{3})^{\top}$. In fact, the parameterisation of the anti-snowflake in this case can be seen to have infinitely many points where it fails to be injective, \Cref{fig:snowflake and antisnowflake} illustrates this phenomenon.
    
    Thus, if a generalised $p$-Koch snowflake does not self-intersect, then it is a Jordan curve. It remains to show that, when this is the case, the snowflake has bounded turning, and is therefore a quasicircle, answering (Q2).

    \begin{theorem}
        For $p\in (\tfrac14,\tfrac13]$ and $\mathbf s, \mathbf t, \mathbf r\in \mathds S$, if the corresponding parameterisation $\gamma$ of  the generalised $p$-Koch snowflake $\mathcal F = \mathcal F(\mathbf s, \mathbf t, \mathbf r)$ from \eqref{eq : parameterisation} is injective on $[0,1)$, then $\mathcal F$ is of bounded turning. 
\end{theorem}

    \begin{proof}
        Let $x,y\in \mathcal F$ with $x \neq y$; indeed, if $x=y$, we can consider $\Gamma = \{x\}$, and hence $\diam(\Gamma) = 0 = \lvert x-y \rvert$. Recall that $\mathcal F$ is the union of the three generalised $p$-Koch curves $\mathcal C(\mathbf s)$, $f_1(\mathcal C(\mathbf r))$ and $f_2(\mathcal C(\mathbf t))$, where we note that $f_1$ and $f_2$ are isometries. We begin with the case where $x$ and $y$ are both contained in the same generalised $p$-Koch curve. Without loss of generality, we can assume that $x,y\in \mathcal C(\mathbf s)$. 
        
        As before, for $k\in \mathbb N$ and $(u_1,...,u_k)\in \{0,1,2,3\}^{k}$, let $V^{u_1,...,u_k}$ denote the rhombus
            \begin{align*}
            \phi_{(u_1,s_0(0))} \circ \phi_{(u_2,s_1(u_1))}\circ\cdots \circ \phi_{(u_{k},s_{k-1}(\sum_{j=1}^{k-1} u_j 4^{k-1-j}))}(V).
            \end{align*}
        For $k=0$ we have $\{0,1,2,3\}^0 = \{\varnothing\}$ where $\varnothing$ denotes the empty word, and set $V^\varnothing = V$. We then write 
        \begin{align*}
            m = \sup\left\{ k \in \mathbb N_0 : x,y \in V^{u_1,...,u_k} \text{ for some } (u_1,...,u_k) \in \{0,1,2,3\}^k\right\}.
        \end{align*}
        We always have $x,y \in V = V^\varnothing$, and so $m \geq 0$, as the set we take the supremum over is non-empty. Furthermore, the supremum is always attained by some integer $k\in \mathbb N_0$ as otherwise there would exist some $(u_k)_{k\in \mathbb N} \in \{0,1,2,3\}^{\mathbb N} $ such that $x$ and $y$ are both contained in the singleton $\bigcap_{k\in \mathbb N} V^{u_1,...,u_k}$, meaning $x=y$. Let $(u_1,...,u_m)\in \{0,1,2,3\}^m$ be such that $x,y\in V^{u_1,...,u_m}$. We now consider two different cases: (i) $x$ and $y$ lie in two adjacent rhombi at level $m+1$, say $x \in V^{u_1,...,u_m,i}$ and $y\in V^{u_1,...,u_m,i+1}$ for some $i\in \{0,1,2\}$, or (ii) either $x\in V^{u_1,...,u_m,0}$ and $y\in V^{u_1,...,u_m,2}\cup V^{u_1,...,u_m,3}$, or $x\in V^{u_1,...,u_m,1}$ and $y\in V^{u_1,...,u_m,3}$.

        We begin with Case (ii). In this case at least one of the rhombi $V^{u_1,...,u_m, 1}$ and $V^{u_1,...,u_m,2}$ lies `in between' $x$ and $y$, giving us the lower bound $\lvert x-y \rvert \geq p^{m}(\tfrac12-p)$, see \Cref{fig : case(ii) rhombus}. By construction there is a subarc $\Gamma$ of $\mathcal C(\mathbf s)$ with endpoints $x$ and $y$ that is entirely contained within $V^{u_1,...,u_m}$, and so $\diam(\Gamma) \leq \diam(V^{u_1,...,u_m}) = p^m$. Therefore, $\diam(\Gamma) \leq K_1 \, \lvert x-y \rvert$ with $K_1=2/(1-2p)$.

        \tikzmath{\x = 1/(2 * sqrt(6)); \y = (7/24) * \x; \z = (1 - 7/24) * \x; } 
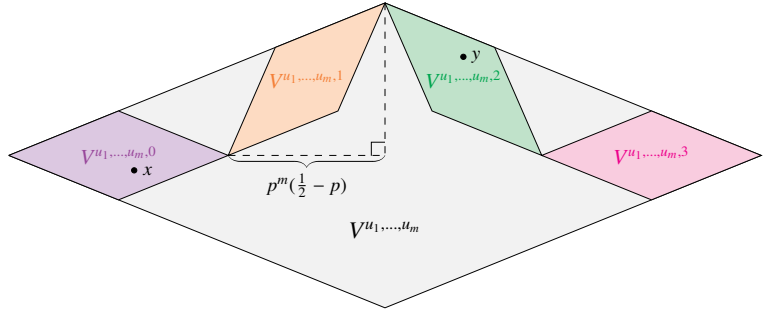
\begin{figure}[ht]
    \centering
    \scalebox{0.90}{
    \begin{tikzpicture}[scale = 11]
        \draw[black,fill=gray!10, -] (-1/2,0) -- (0,\x) -- (1/2,0) -- (0, -\x) -- (-1/2,0); 
        \node at (0,-1/10) {\color{black}$V^{u_1,...,u_m}$};
        \filldraw[fill=Purple!25] (-1/2,0) -- (-17/48,\y) -- (7/24 - 1/2,0) -- (-17/48, -\y) -- (-1/2,0); 
        \node at (7/48-1/2,0) {\small{\color{Purple}$V^{u_1,...,u_m,0}$}};
        \filldraw[fill=Orange!25] (7/24 - 1/2,0) -- (-7/48,\z) -- (0, \x) -- (-1/16, \y) -- (7/24 - 1/2,0); 
         \node at (-5/48,\x/2) {\small{\color{Orange}$V^{u_1,...,u_m,1}$}};
         \filldraw[fill=Green!25] (0, \x) -- (7/48,\z) -- (1/2 - 7/24,0) -- (1/16, \y) -- (0, \x); 
        \node at (5/48,\x/2) {\small{\color{Green}$V^{u_1,...,u_m,2}$}};
         \filldraw[fill=magenta!25] (1/2 - 7/24,0) -- (17/48,\y) -- (1/2,0) -- (17/48, -\y) -- (1/2 - 7/24,0); 
        \node at (1/2-7/48,0) {\small{\color{magenta} $V^{u_1,...,u_m,3}$}};
        
        \draw[dashed] (7/24 - 1/2,0) -- (0,0); 
        \draw [decorate,
        decoration = {calligraphic brace,amplitude=5pt,}] (0,0) --  (7/24 - 1/2,0);
        \draw[dashed] (0,\x) -- (0,0); 

        \node at (7/48 - 1/4,-0.04) {\footnotesize{$p^{m}(\tfrac12-p)$}};

        \filldraw[black] (8/48-1/2,-0.02) circle (0.1pt) node[anchor=west]{\footnotesize{$x$}};
        \filldraw[black] (5/48,\x/2 + 0.03) circle (0.1pt) node[anchor=west]{\footnotesize{$y$}};
        
        \coordinate (A) at (0,0);
        \coordinate (B) at (0,-0.0005);
        \coordinate (C) at (-0.0005,-0.0005);
        \pic [draw, " ", angle radius=0.2cm] {right angle = A--B--C};
    
    \end{tikzpicture}}
    \caption{An example of the rhombus $V^{u_1,...,u_m}$ in Case (ii).}
    \label{fig : case(ii) rhombus}
\end{figure}

\noindent We now move to Case (i), where $x \in V^{u_1,...,u_m,i}$ and $y\in V^{u_1,...,u_m,i+1}$ for some $i\in \{0,1,2\}$. In this case we let $n\geq m$ be the maximum of all $k\in \mathbb N$ for which $x$ and $y$ lie in rhombi with non-empty intersection at level $n$. Such an $n$ exists, as otherwise $x=y$. As we assume that $\rho_{\mathbf s}$ is injective there exists a continuum $\Gamma$ containing $x$ and $y$ that is entirely contained within these two adjacent level $n$ rhombi, which are both of diameter $p^n$, and hence $\diam(\Gamma) \leq 2p^n$.

Lower bounds for the distance $\lvert x-y \rvert$ are then found by considering the worst case, namely when the smallest distance between two disjoint level $n+1$ rhombi is minimal, which is different for the two cases $p = \frac13$ and $p\in (\frac14,\frac13)$, see for instance \Cref{fig : case(i) rhombi}. 

\tikzmath{\p = 1/(2*sqrt(3));  } 
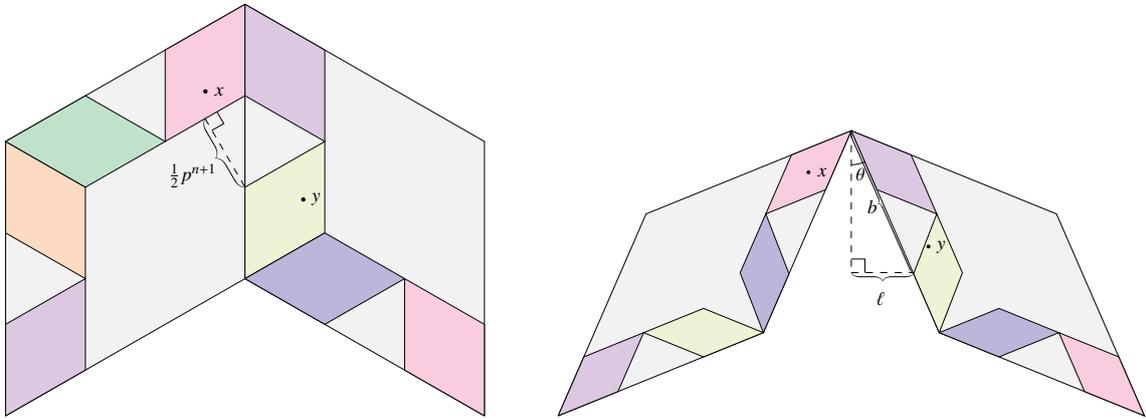
\begin{figure}[h!]
    \centering
    \scalebox{0.90}{
    \begin{tikzpicture}[scale = 7]

        \begin{scope}[rotate=60]
        
        \draw[black,fill=gray!10, -] (0,0) -- (-1/2, \p) -- (-1,0) -- (-1/2, -\p) -- (0,0); 
        \filldraw[fill=Purple!25] (-1,0) -- (-5/6, \p/3) -- (-2/3,0) -- (-5/6, -\p/3) -- (-1,0);
        \filldraw[fill=Orange!25] (-2/3,0) -- (-2/3, 2*\p/3) -- (-1/2,\p) -- (-1/2,\p - 2*\p/3) -- (-2/3,0);
        \filldraw[fill=Green!25] (-1/3,0) -- (-1/3, 2*\p/3) -- (-1/2,\p) -- (-1/2,\p - 2*\p/3) -- (-1/3,0);
        \filldraw[fill=magenta!25] (0,0) -- (-1/6, \p/3) -- (-1/3,0) -- (-1/6, -\p/3) -- (0,0);
        \filldraw[black] (-1/5,-0.02) circle (0.1pt) node[anchor=west]{\footnotesize{$x$}};
        \draw[dashed] (-1/4, -\p/6) -- (-1/3, -2*\p/3);
        \draw [decorate,
        decoration = {calligraphic brace,amplitude=4.5pt,}] (-1/3, -2*\p/3) --  (-1/4, -\p/6);
        \coordinate (A) at (-1/3, -2*\p/3);
        \coordinate (B) at (-1/4, -\p/6);
        \coordinate (C) at (-1/6, -\p/3);
        \pic [draw, " ", angle radius=0.2cm] {right angle = A--B--C};
        \end{scope}
        \node at (-0.11,-0.36) {\footnotesize{$\frac12 p^{n+1}$}};

        \begin{scope}[rotate=-60]
        \draw[black,fill=gray!10, -] (0,0) -- (1/2, \p) -- (1,0) -- (1/2, -\p) -- (0,0); 
        \filldraw[fill=Purple!25] (0,0) -- (1/6, \p/3) -- (1/3,0) -- (1/6, -\p/3) -- (0,0);
        \filldraw[fill=SpringGreen!25] (1/3,0) -- (1/3, -2*\p/3) -- (1/2,-\p) -- (1/2,-\p + 2*\p/3) -- (1/3,0);
        \filldraw[black] (5/12,-0.1) circle (0.1pt) node[anchor=west]{\footnotesize{$y$}};
        \filldraw[fill=Blue!25] (2/3,0) -- (2/3, -2*\p/3) -- (1/2,-\p) -- (1/2,-\p + 2*\p/3) -- (2/3,0);
        \filldraw[fill=magenta!25] (1,0) -- (5/6, \p/3) -- (2/3,0) -- (5/6, -\p/3) -- (1,0);   
        \end{scope}
    \end{tikzpicture}} \hspace{2em}
    \scalebox{0.90}{
    \begin{tikzpicture}[scale = 6]
        \begin{scope}[rotate=44.42]
        \draw[black,fill=gray!10,-] (-1/2-1/2,0) -- (0-1/2,\x) -- (1/2-1/2,0) -- (0-1/2, -\x) -- (-1/2-1/2,0); 
        \filldraw[fill=Purple!25] (-1/2-1/2,0) -- (-17/48-1/2,\y) -- (7/24 - 1/2-1/2,0) -- (-17/48-1/2, -\y) -- (-1/2-1/2,0); 
        \filldraw[fill=SpringGreen!25] (7/24 - 1/2-1/2,0) -- (-7/48-1/2,-\z) -- (0-1/2, -\x) -- (-1/16-1/2, -\y) -- (7/24 - 1/2-1/2,0); 
        \filldraw[fill=Blue!25] (0-1/2, -\x) -- (7/48-1/2,-\z) -- (1/2 - 7/24-1/2,0) -- (1/16-1/2, -\y) -- (0-1/2, -\x); 
        \filldraw[fill=magenta!25] (- 7/24,0) -- (17/48-1/2,\y) -- (0,0) -- (17/48-1/2, -\y) -- (- 7/24,0); 
        \filldraw[black] (-7/48,0) circle (0.1pt) node[anchor=west]{\footnotesize{$x$}};
        \end{scope}

        \begin{scope}[rotate=-44.42]
        \draw[black,fill=gray!10,-] (-1/2+1/2,0) -- (0+1/2,\x) -- (1/2+1/2,0) -- (0+1/2, -\x) -- (-1/2+1/2,0); 
        \filldraw[fill=Purple!25] (-1/2+1/2,0) -- (-17/48+1/2,\y) -- (7/24 - 1/2+1/2,0) -- (-17/48+1/2, -\y) -- (-1/2+1/2,0);
        \filldraw[fill=SpringGreen!25] (7/24 ,0) -- (-7/48+1/2,-\z) -- (1/2, -\x) -- (-1/16+1/2, -\y) -- (7/24,0);
        \filldraw[black] (8/24,-1/14) circle (0.1pt) node[anchor=west]{\footnotesize{$y$}};
        \filldraw[fill=Blue!25] (1/2, -\x) -- (7/48+1/2,-\z) -- (1 - 7/24,0) -- (1/16+1/2, -\y) -- (1/2, -\x); 
        \filldraw[fill=magenta!25] (1/2 - 7/24+1/2,0) -- (17/48+1/2,\y) -- (1/2+1/2,0) -- (17/48+1/2, -\y) -- (1/2 - 7/24+1/2,0);
        \end{scope}

        \draw[black,dashed] (0,0) -- (0,-1/3 - 0.015);
        \draw[black,dashed] (0,-1/3 - 0.015) -- (1/7 + 0.01,-1/3 - 0.015);
        \draw [decorate,
        decoration = {calligraphic brace,amplitude=3.5pt,}] (1/7 + 0.01,-1/3 - 0.015) -- (0,-1/3 - 0.015);
        \node at (1/14,-1/3 - 0.08) {\footnotesize{$\ell$}};
        \draw [decorate,
        decoration = {calligraphic brace,amplitude=2pt,}] (1/7 + 0.01,-1/3 - 0.015) -- (0,0);
        \node at (1/14 - 0.02,-1/6 - 0.02) {\footnotesize{$b$}};
        \coordinate (A) at (0,-1/3 - 0.015);
        \coordinate (B) at (0,0);
        \coordinate (C) at (1/7 + 0.01,-1/3 - 0.015);
        \pic [draw, -, "\footnotesize{$\theta$}", angle eccentricity=1.3, angle radius=0.5cm] {angle = A--B--C};
        \coordinate (A) at (0,0);
        \coordinate (B) at (0,-1/3 - 0.015);
        \coordinate (C) at (1/7 + 0.01,-1/3 - 0.015);
        \pic [draw, " ", angle radius=0.2cm] {right angle = A--B--C};
    \end{tikzpicture}}
    \caption{Case (i) at level $n$ in the worst case when $p=\tfrac13$ (left) and when $p\in (\tfrac14,\tfrac13)$ (right).}
    \label{fig : case(i) rhombi}
\end{figure}

In the case $p= \frac13$, this distance is given by the height of an equilateral triangle whose sides are of length $p^{n+1}\!\sqrt{p}$ (namely the length of the sides of the level $n+1$ rhombi) and so we obtain $\lvert x-y \rvert \geq p^{n+1}\!\sqrt{3p}/2 = p^{n+1}/2 =\frac p4\diam(\Gamma)$. Thus, $\diam(\Gamma)\leq K'_2 \, \lvert x-y \rvert$ with $K'_2=\frac 4p$. In the case $p\in (\frac14,\frac13)$, this minimal distance is bounded from below by the length $\ell = b\cdot\sin(\theta) = p^{n-1}\frac{(1-p)(1-3p)}{2}$ of the side opposing the angle $\theta = \arcsin(\frac{1-3p}{2p\sqrt{p}})$, in a right triangle with a hypotenuse of length $b = p^{n}\sqrt{p}(1-p)$, and so $\lvert x-y \rvert \geq \ell = \diam(\Gamma)\frac{(1-p)(1-3p)}{4p}$ and $\diam(\Gamma) \leq K_2 \, \lvert x-y \rvert$ with $K_2=\frac{4p}{(1-p)(1-3p)}$. 

It remains to check the cases when $x$ and $y$ belong to different generalised $p$-Koch curves. By symmetry, we can assume that $x \in C(\mathbf s)$ and $y \in f_1(\mathcal C(\mathbf t))$. As before let $n$  be the maximum of all $k\in \mathbb N$ for which $x$ and $y$ lie in rhombi with non-empty intersection at level $n$, where this time the rhombus $x$ lies in is from level $n$ of the construction of $\mathcal C(\mathbf s)$, whereas that $y$ lies in is from level $n$ of the construction of $f_1(\mathcal C(\mathbf r))$. Again there then exists a continuum $\Gamma$ in $\mathcal F$ containing $x$ and $y$ that is entirely contained within these two level $n$ rhombi, giving us the bound $\diam(\Gamma) \leq 2p^n$. The two worst cases for $p = \frac13$ and $p \in (\frac14,\frac13)$ are then analogous to those in Case (i) above, see also \Cref{fig : case(i) rhombi}. For $p = \frac13$, since $\gamma$ is assumed to be injective on $[0,1)$, we obtain the same bound $\diam(\Gamma) \leq K'_2 \, \lvert x-y \rvert$ with $K'_2=\frac4p$. For $p \in (\frac14,\frac13)$ we obtain the bound $\lvert x-y \rvert \geq p^{n}(1-p)\sqrt{p} \sin(\varphi)$, where $\varphi = \frac{\pi}{6}-\arccos(\frac{1}{2\sqrt{p}})$. Thus, $\diam(\Gamma)\leq K_3 \, \lvert x-y \rvert$ with $K_3=\frac{8}{1-p}(1 - \sqrt{3(4p-1)})^{-1}$. Therefore, setting 
    \begin{align*}
    K = \begin{cases}
        \max\{K_1,K'_2\}=12 & \text{if} \; p = \frac13,\\
        \max \{ K_1,K_2,K_3\} = \frac{8}{1-p}(1 - \sqrt{3(4p-1)})^{-1} & \text{if} \; p \in (\frac14,\frac13), 
    \end{cases}
    \end{align*}
we obtain that for any $x,y\in \mathcal F$ there exists a continuum $\Gamma$ in $\mathcal F$ containing $x$ and $y$ such that $\diam(\Gamma)\leq K \, \lvert x-y \rvert$, meaning $\mathcal F$ is of bounded turning.
\end{proof}

\section{A full spectrum of areas enclosed by \texorpdfstring{$p$}{p}-Koch snowflakes}\label{sec_area}

In this section, we move to answering (Q3). That is, we study the area enclosed by any given generalised $p$-Koch snowflake $\mathcal{F} = \mathcal{F}(\mathbf s,\mathbf t,\mathbf r)$.  However, before doing so, we make the observation that, as \Cref{fig:random-snowflakes} demonstrates, $\mathcal{F}$ partitions the plane into several connected components; one unbounded connected component and one or more bounded connected components. Therefore, we need a systematic way to determine which of these connected components count as being \textsl{enclosed} by $\mathcal{F}$ and which do not -- we remark here that by \Cref{prop : Jordan curve} it is only in the case when $p=\tfrac13$ that one can obtain more than one bounded component. Loosely speaking, we say that the area of one of the bounded components will contribute to the area enclosed by $\mathcal{F}$ if, as we traverse $\mathcal{F}$ in a counter-clockwise direction, the winding number of each point in the component equals one. Here \textsl{traversing} $\mathcal{F}$ \textsl{in a counter-clockwise direction} means to traverse the snowflake through the parameterisation $\gamma:[0,1]\to \mathcal F$ from \eqref{eq : parameterisation} from $t=0$ to $t=1$. 

With this clarification in hand, let us fix some notation. For $j\in \{0,1,2\}$ let $\mathbf s^j = ((s^j_k(0),s^j_k(1),...,s^j_k(4^{k}-1)))_{k = 0}^\infty \in \mathds{S}$. Let $\mathcal F = \mathcal F(\mathbf s^0,\mathbf s^1, \mathbf s^2) = \mathcal C(\mathbf s^0)\cup f_1(C(\mathbf s^1)) \cup f_2(C(\mathbf s^2))$ be the corresponding generalised $p$-Koch snowflake. Recall, for each $j\in \{0,1,2\}$, that the generalised $p$-Koch curve $\mathcal C(\mathbf s^j)$ is the limit, with respect to the Hausdorff metric, of the curves $\mathcal C_k(\mathbf s^j)$ from \eqref{eq : approximating curve}, as $k$ tends to infinity. Since $f_1$ and $f_2$ are both bi-Lipschitz, it follows that the snowflake $\mathcal F$ is the limit, with respect to the Hausdorff metric, of the curves $\mathcal F_k = \mathcal C_k(\mathbf s^0) \cup f_1(\mathcal C_k(\mathbf s^1)) \cup f_2(\mathcal C_k(\mathbf s^2))$, as $k$ tends to infinity.

For each $k\in \mathbb N_0$, let $a_k$ denote the area of the region enclosed by $\mathcal F_k$, as in \Cref{fig:nolabel3}. Since $\mathcal F_0 = E_0\cup f_1(E_0)\cup f_2(E_0)$ is an equilateral triangle with sides of length $1$, we have $a_0 = \tfrac14\sqrt{3}$. For any $k\in \mathbb N_0$, $\mathcal F_k$ is a closed curve consisting of $3\cdot 4^{k}$ connected line segments of length $p^k$. In the construction of $\mathcal F_{k+1}$, the middle part of each of these line segments is replaced by two legs of an isosceles triangle whose orientation is decided by the values of $s_k^j(i)$ for $j\in \{0,1,2\}$ and $i \in \{0,1,...,4^k-1\}$; if $s_k^j(i) = 0$, the corresponding triangle is pointed outwards with respect to the curve, and if $s_k^j(i) = 1$, it is pointed inwards. Thus, for a fixed $k\in \mathbb N_0$, to obtain $a_{k+1}$ from $a_k$, one subtracts from $a_k$ the area of each of the inward pointing triangles and adds that of each outward pointing triangle. Since each of these triangles is an isosceles triangle with legs of length $p^{k+1}$ and a base of length $p^{k}(1-2p)$, the corresponding area of each triangle is $\frac14 p^{2k}(1-2p)\sqrt{4p-1}$. The value
    \begin{align*}
        \tau_k = \sum_{j=0}^2\#\{ i \in \{0,1,...,4^k-1\} : s_k^j(i) = 0\},\quad k\in \mathbb N_0.
    \end{align*}
is the number of the triangles that are pointed outwards, and thus $3\cdot 4^k - \tau_k$ is the number of triangles that are pointed inwards.
\tikzmath{ \x = sqrt(3)/2;}
\begin{figure}[t]
\centering
    \begin{tikzpicture}[scale = 1.2]
    \draw[black, thick, -, fill=gray] (0,0) -- (3,0) -- (1.5, -3*\x) -- (0,0);
    \draw[black, thick, -, fill=gray] (4,0) -- (5,0) -- (5.5,-\x) -- (6,0) -- (7,0) -- (6.5,-\x) -- (7,-2*\x) -- (6,-2*\x) -- (5.5,-3*\x) -- (5, -2*\x) -- (4,-2*\x) -- (4.5, -\x) -- (4,0);
    \draw[black, thick, -,fill=gray] (8,0) -- (8 + 1/3,0) -- (8.5,\x/3) -- (8 + 2/3, 0) -- (9,0)
    -- (9+1/6,-\x/3) -- (9.5, -\x/3) -- (9+1/3,-\x + \x/3) -- (9.5,-\x)
    -- (9.5+1/6,-\x + \x/3) -- (9.5, -\x/3) -- (10-1/6,-\x/3) -- (10,0)
    -- (10 + 1/3,0) -- (10.5,-\x/3) -- (10 + 2/3, 0) --   (11,0) -- (11-1/6,-\x/3) -- (10.5,-\x/3) -- (10.5+1/6, -\x + \x/3) -- (10.5, -\x) -- (10.5+1/6, -\x-\x/3) -- (11,-2*\x + 2*\x/3) -- (11-1/6, -2*\x +\x/3) --(11,-2*\x) -- (11-1/3, -2*\x) -- (11-1/2, -2*\x - \x/3) -- (11-2/3, -2*\x) -- (10, -2*\x) -- (10-1/6, -2*\x - \x/3) -- (9.5, -2*\x - \x/3) -- (9.5+1/6, -2*\x - 2*\x/3) --  (9.5, -3*\x) -- (9.5-1/6, -2*\x - 2*\x/3) -- (9,-2*\x - 2*\x/3) -- (9+1/6, -2*\x - \x/3) -- (9, -2*\x) -- (8 + 2/3, -2*\x) -- (8.5, -2*\x - \x/3) -- (8+1/3, -2*\x) -- (8,-2*\x) -- (8+1/6, -2*\x +\x/3) -- (8.5, -2*\x + \x/3) -- (8.5 - 1/6, -2*\x + 2*\x/3) -- (8.5,-\x) -- (8.5-1/6, -\x+\x/3) -- (8,-\x +\x/3) -- (8+1/6, -\x/3) -- (8,0);
    \end{tikzpicture}
    \caption{Examples of the areas enclosed by the curves $\mathcal F_0$, $\mathcal F_1$ and $\mathcal F_2$, with $p=\tfrac13$, $\tau_0 = 2$ and $\tau_1 = 8$.}
    \label{fig:nolabel3}
\end{figure}
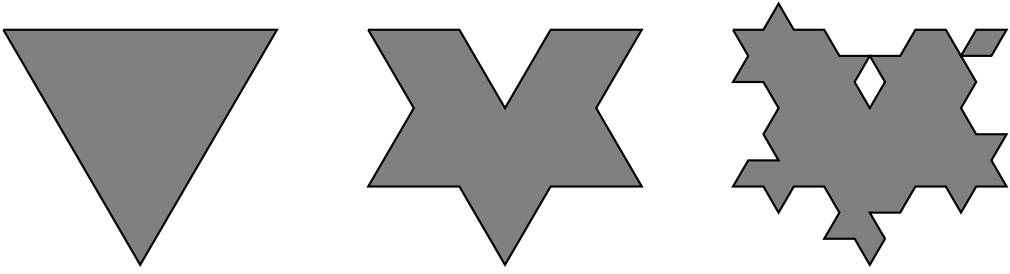
We therefore obtain, for $k\in\mathbb N_0$,
    \begin{align*}
        a_{k+1} = a_k + \tau_k \frac{p^{2k}(1-2p)\sqrt{4p-1}}{4} - (3\cdot 4^{k} - \tau_k)\frac{p^{2k}(1-2p)\sqrt{4p-1}}{4} =  a_{k} + (2\tau_k - 3\cdot 4^k) \frac{p^{2k}(1-2p)\sqrt{4p-1}}{4}.
    \end{align*}
As area is not continuous with respect to the Hausdorff metric, we cannot immediately conclude that the area enclosed by the snowflake $\mathcal F$ is given by $\lim_{k\to\infty} a_k$. To show that this holds, we use that area is monotonic with respect to set-inclusion. Recall that the curve $\mathcal F_k$ consists of $3\cdot 4^k$ line segments, each of length $p^k$, and each of the form $f_i\circ\phi_{\omega_1}\circ\cdots\circ\phi_{\omega_k}(E_0)$ with $i\in\{0,1,2\}$, where $f_0$ denotes the identity, and $\omega_{j} \in A$ for $j\in\{1,\ldots,k\}$. Let $V^{+}$ denote the two line segments joining $P_-$ with $Q_{+}$ and $Q_{+}$ with $P_+$; and let $V^{-}$ denote the two line segments joining $P_-$ with $Q_{-}$ and $Q_{-}$ with $P_+$, where $P_{\pm}$ and $Q_{\pm}$ are as defined in \eqref{eq:main_vertices}. Replacing $E_0$ in each $f_i\circ\phi_{\omega_1}\circ\cdots\circ\phi_{\omega_k}(E_0)$ by $V^+$ we obtain a curve $\mathcal F_k^+$; replacing $E_0$ by $V^-$ we obtain a curve $\mathcal F_k^-$. As the interior of $V$ is a feasible open for the OSC, for each generalised $p$-Koch curve, see \Cref{lemma : well-defined and OSC}, the region enclosed by $\mathcal F$ is contained in the region enclosed by $\mathcal F_k^+$ and contains the region enclosed by $\mathcal F_k^-$. Thus, $a_k^-\leq x\leq a_k^+$ for all $k\in\mathbb N$, where  $a_k^{\pm}$ denotes the area of the region enclosed by $\mathcal F_k^{\pm}$, respectively, and $x$ denotes the area of the region enclosed by $\mathcal F$. As $a_k^+-a_k^-=3\cdot 4^k\cdot p^{2k}\cdot \text{area}(V)$ converges to 0 as $k$ tends to infinity and as $a_k^-\leq a_k\leq a_k^+$, the area $x$ of the region enclosed by the snowflake $\mathcal F$ is
    \begin{align} \label{eq : area snowflake}
        x = \lim_{k\to \infty} a_k = a_0 + \frac{(1-2p)\sqrt{4p-1}}{4}\sum_{k=0}^\infty (2\tau_k - 3\cdot 4^k) p^{2k}.
    \end{align}
In particular, as the classical $p$-Koch snowflake is given by $s^j_k(i) = 0$ for all $j\in \{0,1,2\}$, $k\in \mathbb N_0$ and $i\in \{0,1,...,4^k\}$, $\tau_k = 3\cdot 4^k$ for each $k\in \mathbb N$. Thus, the area $y_p$ of the region enclosed by the classical $p$-Koch snowflake is given by 
    \begin{align*}
        y_p = a_0 + \frac{(1-2p)\sqrt{4p-1}}{4}\sum_{k=0}^\infty 3\cdot 4^k p^{2k} = \frac{\sqrt{3}}{4} + \frac{3}{4}\cdot \frac{\sqrt{4p-1}}{1+2p}.
    \end{align*}
Similarly, the $p$-Koch anti-snowflake is obtained by setting $s^j_k(i) = 1$ for all $j\in \{0,1,2\}$, $k\in \mathbb N_0$ and $i\in \{0,1,...,4^k\}$, hence $\tau_k = 0$ for each $k\in \mathbb N_0$,  and thus the area $x_p$ of the region enclosed by the $p$-Koch anti-snowflake is given by
    \begin{align*}
        x_p = a_0 - \frac{(1-2p)\sqrt{4p-1}}{4}\sum_{k=0}^\infty 3\cdot 4^k p^{2k} =  \frac{\sqrt{3}}{4} - \frac{3}{4}\cdot \frac{\sqrt{4p-1}}{1+2p}.
    \end{align*}
Moreover, we can use \eqref{eq : area snowflake} to prove the following, giving also an affirmative answer to (Q3).

\begin{theorem} \label{theorem:area spectrum}
    Let $p\in (\tfrac14,\tfrac13]$ be fixed. For each $y\in [x_p,y_p]$, there exists a generalised $p$-Koch snowflake $\mathcal F$ such that the area of the region enclosed by $\mathcal F$ equals $y$.
\end{theorem}

\begin{proof}
We can rewrite \eqref{eq : area snowflake} as
    \begin{align*}
        x = x_p +  \frac{(1-2p)\sqrt{4p-1}}{2}\sum_{k=0}^\infty \tau_k p^{2k}.
    \end{align*}
Hence, if we write $\beta = p^{-2} \in [9,16)$, we are done if we can prove that for any $z\in [0,\frac{3}{1-4p^2}]$ there exists a sequence $(\tau_k)_{k\in \mathbb N_0}$ with $\tau_k \in \{0,1,...,3\cdot 4^k\}$ for each $k\in \mathbb N_0$ satisfying $z = \sum_{k=0}^\infty \tau_k\beta^{-k}$. Here we note that
    \begin{align*}
        \frac{2(y_p-x_p)}{(1-2p)\sqrt{4p-1}} = \frac{3}{1-4p^2} = \sum_{m=0}^\infty \frac{3\cdot 4^k}{\beta^k}.
    \end{align*}
The problem has therefore turned from a geometrical one into one concerning $\beta$-expansions for real numbers. For each $n\in \mathbb N_0$ we consider the set
    \begin{align*}
        X_n = \left\{\sum_{k=n}^\infty \frac{\tau_k}{\beta^k} : \tau_k \in \{0,1,...,3\cdot 4^{k}\} \text{ for all } k\geq n \right\},
    \end{align*}
and observe $X_n =\bigcup_{k = 0}^{3\cdot 4^{n}} \{k\beta^{-n} + x : x\in X_{n+1}\} = \bigcup_{k = 0}^{3\cdot 4^{n}} (k\beta^{-n} + X_{n+1})$ for all $n\in \mathbb N_0$. It is known, see for instance \cite{MR1917322}, that for any $\beta>1$,
    \begin{align*}
        [0,1] \subseteq \left\{\sum_{m=1}^\infty \frac{\tau_m}{\beta^m} : \tau_m \in \{0,1,...,\lceil \beta\rceil -1\} \text{ for all } m\in \mathbb N \right\}.
    \end{align*}
Since we have $\beta \in [9,16)$, we have that $\lceil \beta \rceil -1 \leq 15 < 3\cdot 4^n$ for each $n\in \mathbb N_{\geq 2}$. It therefore follows that $[0,\tfrac{1}{\beta^{n-1}}] \subseteq X_n$ for each $n\in \mathbb N_{\geq 2}$. Equivalently, we have, for each $n\in \mathbb N$, that $[0,\tfrac{1}{\beta^n}] \subseteq X_{n+1}$, and hence
    \begin{align*}
        X_n = \bigcup_{k=0}^{3\cdot 4^n} \left( \frac{k}{\beta^n} + X_{n+1}\right) \supseteq \bigcup_{k=0}^{3\cdot 4^n} \left( \frac{k}{\beta^n} + \Big[0, \frac{1}{\beta^n}\Big]\right) = \Big[0, \frac{3\cdot 4^n + 1}{\beta^n}\Big].
    \end{align*}
In particular, we have, for each $n\in \mathbb N$, that $[0,\frac{3\cdot 4^{n+1} + 1}{\beta^{n+1}}] \subseteq X_{n+1}$, and hence
    \begin{align*}
        X_n = \bigcup_{k=0}^{3\cdot 4^n} \left( \frac{k}{\beta^n} + X_{n+1}\right) \supseteq \bigcup_{k=0}^{3\cdot 4^n} \left( \frac{k}{\beta^n} + \Big[0, \frac{3\cdot 4^{n+1} + 1}{\beta^{n+1}}\Big]\right) = \Big[0, \frac{3\cdot 4^n}{\beta^n} + \frac{3\cdot 4^{n+1}+ 1}{\beta^{n+1}}\Big].
    \end{align*}
Repeating the above, yields $[0, \sum_{k=n}^m \frac{3\cdot 4^k}{\beta^k} + \frac{1}{\beta^m}] \subseteq X_n$, for $n, m \in \mathbb N$ with $m > n$. Letting $m$ tend infinity, we obtain 
    \begin{align*}
        \Big[0, \sum_{k=n}^\infty \frac{3\cdot 4^k}{\beta^k} \Big) \subseteq \bigcup_{m=n+1}^\infty \Big[0, \sum_{k=n}^m \frac{3\cdot 4^k}{\beta^k} + \frac{1}{\beta^m} \Big] \subseteq X_n.
    \end{align*}
Taking $\tau_k = 3\cdot 4^k$, for each $k\geq n$, reveals that $\sum_{k=n}^\infty 3\cdot 4^k\beta^{-k} \in X_n$ for $n\in \mathbb N$, and so $X_n \supseteq [0,\sum_{k=n}^\infty 3\cdot 4^k\beta^{-k}]$ for each $n\in \mathbb N$. In particular, we have $X_1 \supseteq [0, \sum_{k=1}^\infty 3\cdot 4^k\beta^{-k} ] =  [0, \frac{12}{\beta -4}]$.  Noting that $\beta < 16$ and hence $\frac{12}{\beta -4} > 1$, we conclude that $X_0 \supseteq \bigcup_{k=0}^3[ k, k+\frac{12}{\beta-4}] = [0, 3+\frac{12}{\beta-4}] = [0, \frac{3}{1-4p^2}]$.
\end{proof}

Our discussion preceding \Cref{theorem:area spectrum} highlights that, fixing $p \in (\frac{1}{4}, \frac{1}{3}]$, there exists only one generalised $p$-Koch snowflake with enclosed area $x_{p}$, namely the $p$-Koch anti-snowflake, and only one generalised $p$-Koch snowflake with enclosed area $y_{p}$. Moreover, for $y\in (x_p,y_p)$, \Cref{theorem:area spectrum} gives the existence of a generalised $p$-Koch snowflake with enclosed area
    \begin{align*}
        y=x_p + \frac{(1-2p)\sqrt{4p-1}}{2}\sum_{k=0}^\infty \tau_k p^{2k},
    \end{align*}
where the sequence $(\tau_k)_{k\in\mathbb N_0}$ with $\tau_k\in\{0,\ldots,3\cdot 4^k\}$ is witnessed in the snowflake's construction. The following result (\Cref{EJK-thm-gen}, for which we will require \Cref{EJK-lem-gen}) -- the proof of which is an adaptation of a beautiful proof of Erd\"os, Jo\'o and Komornik \cite{EJK1990} that for $\beta \in (1,\frac{1+\sqrt{5}}{2})$, every $x \in (1, \frac{1}{\beta-1})$ has uncountably many $\beta$-expansions -- yields that, given $y \in (x_{p}, y_{p})$, there exist uncountably many generalised $p$-Koch snowflakes whose enclosed area is $y$. 

\begin{lemma}\label{EJK-lem-gen}
If $p \in (\frac14, \frac13]$ and $k \geq 2$ is an integer satisfying 
    \begin{align}\label{eq:EJK-k-bound}
        \frac{(4 p^{2})^{k}}{1-(4 p^{2})^{k}} < \frac{4p^{2}}{1-4p^{2}} - \frac{1}{3},
    \end{align}
then, for $m \in \mathbb{N}$,
    \begin{align*}
    \sum_{\substack{j \in \mathbb{N} \setminus k \mathbb{N},\\ j > m}} 3 \cdot 4^{j} p^{2j}
    \geq p^{2m}.
    \end{align*}
\end{lemma}

\begin{proof}
Fix $k \geq 2$ and $m \in \mathbb{N}$. Let $n \in \mathbb{N}_{0}$ and $l \in \{1, 2, \dots, k-2\}$ be such that  $m = n k + l$. Then, we have that
    \begin{align*}
    \sum_{\substack{j \in \mathbb{N} \setminus k \mathbb{N},\\ j > m}} 3 \cdot (4 p^{2})^{j}
    &= 3 \sum_{\substack{j \in \mathbb{N} ,\\ j > m}} (4 p^{2})^{j} - 3 \sum_{\substack{j \in \mathbb{N} ,\\ j \geq n+1}} (4 p^{2})^{jk}\\
    &= \frac{3(4 p^{2})^{m+1} }{1-4p^{2}} - \frac{ 3(4 p^{2})^{(n+1)k} }{1-(4p^{2})^{k}}\\
    &= p^{2m} \cdot 3 \cdot 4^{m+1} p^{2} \left( \frac{1}{1-4p^{2}} - \frac{(4 p^{2})^{k-l-1} }{1-(4p^{2})^{k}} \right)\\
    &\geq p^{2m} 4^{m+2} p^{2} (1 - (4 p^{2})^{k-l-1})
    \geq p^{2m} 4^{m+2} p^{2} (1 - 4 p^{2})
    \geq p^{2m}.
\intertext{Moreover, if $k \geq 2$ is an integer satisfying \eqref{eq:EJK-k-bound}, and if $n \in \mathbb{N}$ is such that $m = nk-1$, then}
    \sum_{\substack{j \in \mathbb{N} \setminus k \mathbb{N},\\ j > m}} 3 \cdot  (4 p^{2})^{j}
    &= 3 \sum_{\substack{j \in \mathbb{N},\\ j \geq m+2}} (4 p^{2})^{j} - 3 \sum_{\substack{j \in \mathbb{N},\\ j \geq n+1}} (4 p^{2})^{jk}\\
    &= \frac{3 (4 p^{2})^{m+2} }{1-4p^{2}} - \frac{3 (4 p^{2})^{(n+1)k} }{1-(4p^{2})^{k}}\\
    &= \frac{3 (4 p^{2})^{m+2} }{1-4p^{2}} - \frac{3 (4 p^{2})^{m+1+k} }{1-(4p^{2})^{k}}\\
    &= p^{2m} \cdot 3 \cdot 4^{m+1} p^{2} \left( \frac{4 p^{2} }{1-4p^{2}} - 
    \frac{(4 p^{2})^{k} }{1-(4p^{2})^{k}} \right)
    \geq p^{2m}.\qedhere
    \end{align*}
\end{proof}

\begin{theorem}\label{EJK-thm-gen}
For $p \in (\frac{1}{4}, \frac{1}{3}]$ and $y \in (0, \frac{3}{1-4p^2})$, there exist uncountably many $(\nu_{m})_{m \in \mathbb{N}_{0}} \in \prod_{m \in \mathbb{N}_{0}}\{ 0, 1, \dots, 3 \cdot 4^{m}\}$ with
    \begin{align*}
        y = \sum_{m \in \mathbb{N}_{0}} \nu_{m} p^{2m}.
    \end{align*}
\end{theorem}

\begin{proof}
Since $p \in (\frac{1}{4}, \frac{1}{3}]$, we have  $3\sum_{m \in \mathbb{N}} 4^{m} p^{2m} > 1$. Thus, if we can show for $y' \in (0, 3\frac{ 4 p^{2}}{1-4p^2})$, there exist uncountably many $(\nu_{m})_{m \in \mathbb{N}} \in \prod_{m \in \mathbb{N}}\{ 0, 1, \dots, 3 \cdot 4^{m}\}$ with 
    \begin{align}\label{eq:uncountable_rep}
    y' = \sum_{m \in \mathbb{N}} \nu_{m} p^{2m},
    \end{align}
then we can set $v_0=0$ if $y \in (0, 3\frac{ 4 p^{2}}{1-4p^2})$, and $v_0=j$ if $y\in [\frac{(v_{0}-1) - (4-v_{0}) 4 p^{2}}{1-4p^2},\frac{v_{0} - (3-v_{0}) 4 p^{2}}{1-4p^2})$ for $j \in \{1, 2, 3\}$. To this end, observe, since $y' \in (0, 3\frac{4 p^{2}}{1-4p^2})$, we may choose $k\in \mathbb{N}$ as in \Cref{EJK-lem-gen}, sufficiently large so that
    \begin{align*}
        3\sum_{m \in \mathbb{N}} 4^{mk} p^{2mk}
        < y' 
        < 3\sum_{m \in \mathbb{N}} 4^{m} p^{2m} -3\sum_{m \in \mathbb{N}} 4^{mk} p^{2mk}
        = 3\sum_{m \in \mathbb{N}\setminus k\mathbb N} 4^{m} p^{2m}.
    \end{align*}
Now for any choice of $(\delta_{mk})_{m \in \mathbb{N}} \in \prod_{m\in\mathbb N}\{ 0, \dots, 3\cdot 4^{mk}\}$, of which there are uncountably many, we have
    \begin{align}\label{eq:definition_z_uncountable}
        0 <  \underbrace{y' - \sum_{m \in \mathbb{N}} \delta_{mk}p^{2mk}}_{=z}
        < 3\sum_{m \in \mathbb{N}\setminus k\mathbb N}  4^{m} p^{2m}.
    \end{align}
Defining $(\nu_{m})_{m \in \mathbb{N}\setminus k \mathbb{N}} \in \prod_{m \in \mathbb{N}\setminus k \mathbb{N}}\{ 0, 1, \dots, 3 \cdot 4^{m}\}$ inductively as follows, set $\nu_{1} = \max \{ \eta \in \{0,1, \dots, 3 \cdot 4 \} : z < \eta p^{2}\}$, and for $m \in \mathbb{N} \setminus k \mathbb{N}$ with $m \geq 2$, let
    \begin{align}\label{eq:EJK-cond3}
    \nu_{m} = \max\left\{ \eta \in \{0,1,\dots,3\cdot4^{m}\} : \eta p^{2m} + \sum_{\substack{j \in \mathbb{N} \setminus k \mathbb{N},\\ j < m}} \nu_{j}\, p^{2j} < z \right\},
    \end{align}
we claim $z = \sum_{m \in \mathbb{N}\setminus k\mathbb{N}} \nu_{m} p^{2m}$.  For if not, then, by construction, 
    \begin{align}\label{EJK-cond4}
    \sum_{m \in \mathbb{N}\setminus k\mathbb{N}} \nu_{m} p^{2m} < z.    
    \end{align}
By \eqref{eq:definition_z_uncountable}, the set $Q = \{ i \in \mathbb{N}\setminus k \mathbb{N} : \nu_{i} \neq 3\cdot 4^{i} p^{2i} \}$ is non-empty and bounded, which can be seen as follows. If $Q = \emptyset$, then $\nu_{i} = 3\cdot 4^{i} p^{2i}$ for all $i \in \mathbb{N} \setminus k\mathbb{N}$, contradicting the upper bound given in \eqref{eq:definition_z_uncountable}, and if $Q$ was unbounded, then
    \begin{align*}
    0 \leq z - \sum_{\substack{j \in \mathbb{N} \setminus k \mathbb{N},\\ j \leq m}} \nu_{j}\, p^{2j} \leq p^{2m}
    \end{align*}
for infinitely many $m \in \mathbb{N}$, contradicting \eqref{EJK-cond4}. Letting $q = \sup Q$, we observe that $\nu_{i} = 3\cdot 4^{i}$ for all $i \in \mathbb{N} \setminus k\mathbb{N}$ with $i > q$, and thus, \eqref{EJK-cond4} in tandem with \Cref{EJK-lem-gen}, implies
    \begin{align*}
    \sum_{\substack{j \in \mathbb{N} \setminus k \mathbb{N},\\ j < q}} \nu_{j}\, p^{2j} + (\nu_q+1)p^{2q}=
    \sum_{\substack{j \in \mathbb{N} \setminus k \mathbb{N},\\ j \leq q}} \nu_{j}\, p^{2j} + p^{2 q}
    \leq
    \sum_{\substack{j \in \mathbb{N} \setminus k \mathbb{N},\\ j \leq q}} \nu_{j}\, p^{2j} + 3\sum_{\substack{j \in \mathbb{N} \setminus k \mathbb{N},\\ j > q}} 4^{j} p^{2j} < z.
    \end{align*}
yielding a contradiction to \eqref{eq:EJK-cond3}.  This proves the claim, and so,
    \begin{align*}
        y
        = \sum_{m \in \mathbb{N}\setminus k\mathbb{N}} \nu_{m} p^{2m} + \sum_{m\in \mathbb{N}}\delta_{mk}p^{2mk}.
    \end{align*}
Since we have uncountably many choices for $(\delta_{mk})_{m \in \mathbb{N}} \in \prod_{m\in\mathbb N}\{ 0, \dots, 3\cdot 4^{mk}\}$, we may construct uncountably many sequences $(\nu_{m})_{m \in \mathbb{N}} \in \prod_{m \in \mathbb{N}}\{ 0, 1, \dots, 3 \cdot 4^{m}\}$ satisfying \eqref{eq:uncountable_rep}.
\end{proof}

\section{The double-sided \texorpdfstring{$p$}{p}-Koch curve and snowflake}\label{sec3}

We have seen that, for each real number $y$ between $x_{p}$ and $y_{p}$, there exists a generalised $p$-Koch snowflake with enclosed area $y$. One may therefore wonder whether the generalised $p$-Koch snowflakes fill the entire space between these two snowflakes. That is, is the union of all generalised $p$-Koch snowflakes equal to the entire region enclosed by the classical $p$-Koch snowflake minus the region enclosed by the $p$-Koch anti-snowflake? 

In order to answer this question, and thus (Q4), we require some preliminary notation and auxiliary results. To this end, we again let $\mathds S$ denote the collection of all sequences $\mathbf s = ((s_k(0),s_k(1),...,s_k(4^{k}-1)))_{k = 0}^\infty $ of tuples satisfying $s_k(i) \in \{0,1\}$ for each $k\in \mathbb N_0$ and $i\in \{0,1,...,4^{k}-1\}$, and let $A = \{0,1,2,3\}\times \{0,1\}$.

\begin{lemma}
    \; $\bigcup_{\mathbf s \in \mathds S} \Sigma(\mathbf s) = A^{\mathbb N}$, where $\Sigma(\mathbf s)$ is as defined in \eqref{eq:Sigma_s}.
\end{lemma}

\begin{proof}
    By construction, $\bigcup_{\mathbf s \in \mathds S} \Sigma(\mathbf s) \subseteq A^{\mathbb N}$. Thus, it remains to prove the reverse inclusion. To this end, let $\omega \in A^{\mathbb N}$ and write $\omega = ((u_k,i_k))_{k\in \mathbb N}$, where $u_k \in \{0,1,2,3\}$ and $i_k\in \{0,1\}$ for each $k\in \mathbb N$. Choosing $\mathbf s = ((s_k(0),...,s_k(4^k - 1))_{k=0}^\infty$ in $\mathds{S}$ with $s_0(0) = i_1$ and $s_{k-1}(\sum_{j=1}^{k-1} u_j4^{k-1-j})  =i_k$ for $k\in \mathbb N_{\geq 2}$, and letting $\mathbf{u} = (u_k)_{k\in \mathbb N}$, we have $\omega = \omega(\mathbf{u}, \mathbf s) \in \Sigma(\mathbf s)$.
\end{proof}

If we let $\mathcal D$ denote the union of all generalised $p$-Koch curves, by construction, we have that
    \begin{align}\label{eq : double-sided}
        \mathcal D = \bigcup_{\mathbf s \in \mathds S} \mathcal C(\mathbf s) = \bigcup_{\mathbf s \in \mathds S} \bigcup_{(\omega_k)_{k\in \mathbb N} \in \Sigma(\mathbf s)}
        \bigcap_{k\in \mathbb N} \left(\phi_{\omega_1}\circ \phi_{\omega_2} \circ \cdots \circ \phi_{\omega_k}\right)(V) = \bigcup_{(\omega_k)_{k\in \mathbb N} \in A^{\mathbb N}}
        \bigcap_{k\in \mathbb N} \left(\phi_{\omega_1}\circ \phi_{\omega_2} \circ \cdots \circ \phi_{\omega_k}\right)(V). 
    \end{align}
Since $\phi_{(0,0)} = \phi_{(0,1)}$ and $\phi_{(3,0)} = \phi_{(3,1)}$, we can replace $A$ by the set $B = A\setminus\{(0,1),(3,1)\}$ in \eqref{eq : double-sided} to obtain
    \begin{align*}
        \mathcal D = \bigcup_{(\omega_k)_{k\in \mathbb N} \in B^{\mathbb N}}
        \bigcap_{k\in \mathbb N} \left(\phi_{\omega_1}\circ \phi_{\omega_2} \circ \cdots \circ \phi_{\omega_k}\right)(V).
    \end{align*}
as with the generalised $p$-Koch curves, we can also construct $\mathcal D$ as the limit, with respect to the Hausdorff metric, of the sets $\mathcal D_k = \bigcup_{\omega_1,...,\omega_k \in B} \phi_{\omega_1}\circ \cdots \phi_{\omega_k}(E_0)$ for $k\in \mathbb N$, where $E_0$ is again the line segment $E_0 = [-\tfrac12,\tfrac12]\times \{0\}$. We see that this process of generating $\mathcal D$ is analogous to that of generating a generalised $p$-Koch curve, but rather than replacing the middle segment of a line segment by two sides of a triangle that is either pointed upward or downward with respect to the orientation of the line segment, we replace it by two sides of both of these triangles. Because of this, we refer to $\mathcal D$ as the \textbf{double-sided $p$-Koch curve}. Note that this name is somewhat misleading, as unlike the generalised $p$-Koch curves, $\mathcal D$ itself is not a curve.

\begin{wrapfigure}[12]{r}{0.45\textwidth}
\vspace{-2em}
  \begin{center}
        \includegraphics[width=0.675\linewidth]{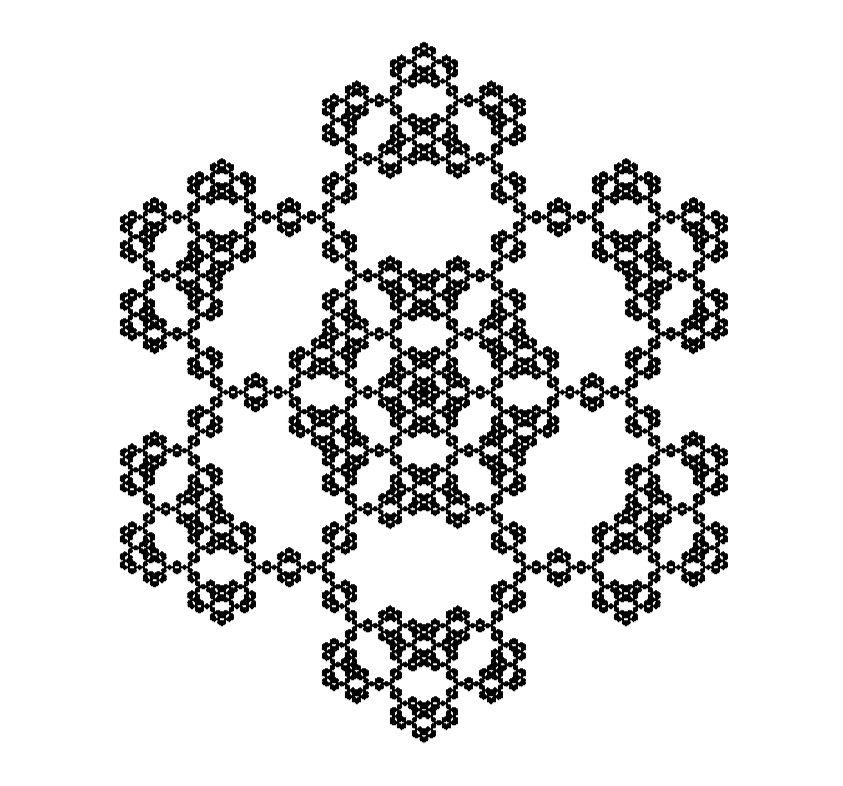}
  \end{center}
  \vspace{-1em}
    \caption{The double-sided $p$-Koch snowflake with $p=\tfrac13$.}
    \label{fig:enter-label4}
\end{wrapfigure}
We call $\mathcal{DF} = \bigcup_{\mathbf s,\mathbf t,\mathbf r \in \mathds S} \mathcal F(\mathbf s,\mathbf t, \mathbf r) = \mathcal D \cup f_1(\mathcal D) \cup f_2(\mathcal D)$, the union of all generalised $p$-Koch snowflakes, the \textbf{double-sided $p$-Koch snowflake}, see \Cref{fig:enter-label4}. Thus, (Q4) is equivalent to asking whether the double-sided $p$-Koch snowflake $\mathcal{DF}$ covers the entire region between the $p$-Koch snowflake and the $p$-Koch anti-snowflake. Since this region has an area of $y_p - x_p$, for $\mathcal{DF}$ to cover this region, it must also contain an open set.

As with the $p$-Koch curve, but unlike with generalised \mbox{$p$-Koch} curves, $\mathcal D$ is the unique attractor set of the iterated function system $\{\phi_{(a,b)} : (a,b)\in B\}$. Since this iterated function system consists of six similarity maps, each with contraction ratio $p$, and satisfies the open set condition by \Cref{lemma : well-defined and OSC}, it follows from \cite[Theorem 5.3(1)]{hutchinson1981fractals} that the Hausdorff dimension of the double-sided Koch curve $\mathcal D$ is given by 
    \begin{align*}
        \dim_{\mathcal{H}}(\mathcal D) = \frac{\log(6)}{\log(p^{-1})}.
    \end{align*}
Consequentially, for $p\in (\tfrac14,\tfrac13]$, we have $\dim_{\mathcal{H}}(\mathcal{DF}) = \dim_{\mathcal{H}}(\mathcal D) < 2$. Therefore, it follows that $\mathcal{DF}$ cannot contain any open subset of $\mathbb R^2$ and thus does not cover any region of positive area.

\section*{Acknowledgements}

This work began as a research experience with undergraduates (REUs) project at the University of Birmingham with K.\,Chen and D.\,Wang, and we are very grateful to them for their insights and helpful discussions. S.\,van\,Golden also acknowledges EPSRC grants EP/S02297X/1 for financial support during the development of this article.

\bibliographystyle{alpha}
\bibliography{bib}

\end{document}